\documentclass[12pt]{article}

\usepackage{setspace}
\usepackage[left=1in,right=1in,top=1in,bottom=1in]{geometry}
   
\usepackage{amssymb,dsfont,amsthm,amsmath,amsfonts,amstext}
\usepackage{float,graphicx,caption,xcolor}
\usepackage{bm}
\usepackage{wrapfig}
\usepackage[colorlinks,citecolor=blue,urlcolor=blue]{hyperref}
\usepackage{caption}
\theoremstyle{plain}
\newtheorem{thm}{Theorem}
\newtheorem{proposition}{Proposition}
\newtheorem{lemma}{Lemma}%[section]
\newtheorem{definition}{Definition}

\usepackage{authblk}

\usepackage{booktabs}
\usepackage{mathrsfs}
\usepackage{enumitem}
\usepackage{array}
\usepackage{mathtools}

\providecommand{\E}{\mathbb{E}}

\providecommand{\cB}{\mathcal{B}}
\providecommand{\cA}{\mathcal{A}}

\providecommand{\R}{\mathbb{R}}
\providecommand{\scrP}{\mathscr{P}}
\providecommand{\calP}{\mathcal{P}}

\newcommand{\bcalP}{\widehat{\mathcal{P}}}
\providecommand{\supp}{\mathrm{supp}\;}

\providecommand{\Lip}{\mathrm{Lip}}

\providecommand{\wass}{\mathsf{W}}
\providecommand{\norm}[1]{\lVert#1\rVert}
\providecommand{\homega}{\widehat{\omega}}
\providecommand{\hpi}{\widehat{\pi}}
\providecommand{\hOmega}{\widehat{\Omega}}

\providecommand{\profile}{\mathsf{F}}

\providecommand{\limsup}{\lim\sup}
\providecommand{\ind}{\mathbb{I}}
\newcommand{\TV}{\mathsf{TV}}

\DeclareMathOperator*{\argmin}{arg\,min}

\definecolor{longhorn}{rgb}{0.8, 0.33, 0.0}

\begin{document}

% \begin{frontmatter}
\title{Robust Hypothesis Testing with Wasserstein Uncertainty Sets}
\date{}
%\runtitle{Data-driven Robust Hypothesis Testing with \\Wasserstein Uncertainty Sets}
%\thankstext{T1}{Footnote to the title with the ``thankstext'' command.}

\author[1]{ Liyan Xie}
\author[2]{ Rui Gao}
\author[1]{ Yao Xie}

\affil[1]{\footnotesize School of Industrial and Systems Engineering, Georgia Institute of Technology}
\affil[2]{\footnotesize Department of Information, Risk and Operations Management, University of Texas at Austin}

\maketitle
\date{}

\begin{abstract}
We consider a data-driven robust hypothesis test where the optimal test will minimize the worst-case performance regarding distributions that are close to the empirical distributions with respect to the Wasserstein distance. This leads to a new non-parametric hypothesis testing framework based on distributionally robust optimization, which is more robust when there are limited samples for one or both hypotheses. Such a scenario often arises from applications such as health care, online change-point detection, and anomaly detection. We study the computational and statistical properties of the proposed test by presenting a tractable convex reformulation of the original infinite-dimensional variational problem exploiting Wasserstein's properties and characterizing the radii selection for the uncertainty sets. We also demonstrate the good performance of our method on synthetic and real data.
\end{abstract}

%\begin{keyword}[class=MSC]
%\kwd[Primary ]{60K35}
%\kwd{60K35}
%\kwd[; secondary ]{60K35}
%\end{keyword}

% \begin{keyword}
% \kwd{Distributionally robust optimization}
% \kwd{Hypothesis testing}
% \end{keyword}

% \end{frontmatter}

\section{Introduction}

Hypothesis testing is a fundamental problem in statistics and an essential building block for machine learning problems such as classification and anomaly detection. The goal of hypothesis testing is to find a decision rule to discriminate between two hypotheses given new data while achieving a small probability of errors. However, the exact optimal test is difficult to obtain when the underlying distributions are unknown. This issue is particularly challenging when the number of samples is limited, and we cannot obtain accurate estimations of the distributions. The limited sample scenario (for one or both hypotheses) commonly arises in many real-world applications such as medical imaging diagnosis \cite{aresta2019bach}, online change-point detection \cite{poor2008quickest}, and online anomaly detection \cite{chandola2010anomaly}.

\subsection{Why distribution-free minimax test}

For hypothesis testing, the well-known Neyman-Pearson Lemma \cite{neyman1933ix} establishes that the likelihood ratio gives the optimal test for two simple hypotheses. This requires to specify {\it a priori} two true distribution functions $P_1$ and $P_2$ for the two hypotheses, which, however, are usually unknown in practice. When the assumed distributions deviate from true distributions, the likelihood ratio test may experience a significant performance loss.

Typically there are ``training'' samples available for both hypotheses. A commonly used approach is the generalized likelihood ratio test (GLRT), which assumes parametric forms for the distributions and estimates parameters using data and plug into the likelihood ratio statistic. Another popular method is the density ratio estimation \cite{DensityRatio2012}. However, in many scenarios, the training samples for one or both hypotheses can be small. For instance, we tend to have a small sample size for patients in healthcare applications. In limited-sample scenarios, it can be challenging to estimate parameters for GLRT (especially in the high dimensional case) or to estimate density ratios accurately. Without reliable estimation of the underlying distributions, various forms of robust hypothesis testing \cite{huber1965robust, huber1973minimax,Levy2009,GulZoubir2017} have been developed by considering different ``uncertainty sets''. Huber's seminar work \cite{huber1965robust} sets the uncertainty set as the $\epsilon$-contamination sets that contain distributions close to a nominal distribution defined by total-variation distance. In \cite{huber1973minimax}, the optimal tests are characterized under majorization conditions, which, however, are intractable in general. Thus, there remains a computational challenge to find the optimal test, especially when the data is \textit{multi-dimensional}. This has become a significant obstacle in applying robust hypothesis tests in practice.

We consider a setting where the sample size is small. When there are limited samples, the empirical distribution may have ``holes'' in the sample space: places where we do not have samples yet, but there is a non-negligible probability for the data to occur, as illustrated in Figure~\ref{fig:LFD_illustration}. Thus, we may not want to restrict the true distribution to be on the same support of the empirical distribution. However, many commonly used distance divergences for probability distributions, such as Kullback-Leibler divergence,  are defined for distributions with common support. Thus, in our setting, it can be restrictive if we were to construct uncertainty sets using the Kullback-Leibler divergence (e.g., \cite{Levy2009} and \cite{GulZoubir2017}). Similarly, total-variational norm-induced uncertainty sets will have this issue since they encourage distributions with the same support as the nominal distribution. This motivates us to consider an uncertainty set formed by the Wasserstein distance. It measures the distance between distributions using optimal transport metric, which is more suitable for distributions without common support.  

\begin{figure}[!ht]
\begin{center}
\includegraphics[width=1\textwidth]{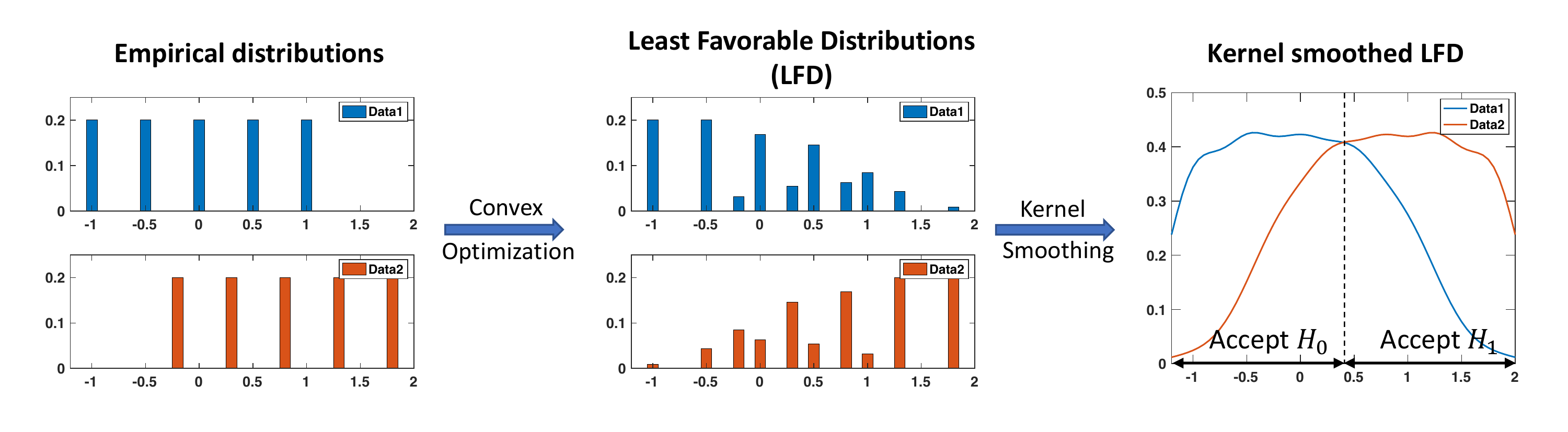}
\captionof{figure}{\small Left: Empirical distributions of two sets of training samples (5 samples each), generated from $\mathcal N(0,1)$ and $\mathcal N(2,1.2)$, respectively. Middle: Least Favorable Distributions (LFD) solve from Lemma~\ref{lemma:wass vs wass} with radius equal to $0.1$. Right: Kernel smoothed versions of LFD (with kernel bandwidth $h =0.3$). }
\label{fig:LFD_illustration}
\end{center}
% \vspace{-0.2in}
\end{figure}

\subsection{Contributions}

In this paper, we present a new non-parametric minimax hypothesis test assuming the distributions under each hypothesis belong to two disjoint ``uncertainty sets'' constructed using the Wasserstein distance. Specifically, the uncertainty sets contain all distributions close to the empirical distributions formed by the training samples in Wasserstein distance. This approach is more robust in small-sample-size regimes when we cannot estimate the true data-generating distributions accurately.

A notable feature of our approach is the computational tractability and explicit characterization of the optimal test. The optimal test is based on a pair of least favorable distributions (LFD) from the uncertainty sets, which is a reminiscence of Huber's robust test. However, here the optimal test form is different, and our LFDs are computationally tractable in general.  An outstanding challenge in finding the minimax test is that we face an infinite-dimensional optimization problem (finding the saddle point for optimal test and LFDs), which is hard to solve in general. To tackle the challenge, we make a connection to recent advances in distributionally robust optimization. In particular, we decouple the original minimax problem into two sub-problems using strong duality, which enable us first to find the optimal test for a given pair of distribution $P_1$ and $P_2$, and then find the LFDs $P^\ast_1$ and $P^\ast_2$ by solving a finite-dimensional convex optimization problem. We further characterize the robust optimal test and extend the test to the ``batch'' setting containing multiple test samples.
 
We also characterize the radii choice of the uncertainty sets, which is an important question that affects the optimal test's generalization property. We prove a theoretical upper bound for the sufficient radii based on the so-called profile function that is defined as the minimum Wasserstein distance between the empirical distributions and distributions that yields the same test as the oracle one. Compared with the commonly used approach in distributionally robust optimization -- the uncertainty set must contain the true distribution, our results shows a matching order that can be attained in worst-case. On the other hand, our results show the advantage of providing the explicit constant term that depends on the densities of the underlying true data-generating distributions.  
% However, it is well-known that empirical Wasserstein distance converges slowly. Under such an assumption, the radius will be unnecessarily large and on the order of $\mathcal O(n^{\scriptscriptstyle -1/d})$, leading to loss of testing power. Instead of assuming the uncertainty sets must contain the {\it true} distributions, we only require the uncertainty sets to include distributions that lead to a test similar to the {\it oracle test} -- the test formed by the true distribution. This can possibly reduce the optimal radii to be on the order of $\mathcal O(n^{\scriptscriptstyle -1/2})$ under certain constraints on the true data distributions. 
%which is independent of the data dimension. 

Finally, we show our method's good performance using simulated and real data, and demonstrate its applicability for sequential human activity detection.

\subsection{Related work}

Robust hypothesis testing has been developed under the minimax framework by considering various forms of ``uncertainty sets''. Seminal work by Huber \cite{huber1965robust} considers the $\epsilon$-contamination sets that contain distributions close to a nominal distribution defined by total-variation distance. Huber and Strassen later generalized the results in \cite{huber1973minimax} based on the observation that the $\epsilon$-contamination sets can be described using the so-called alternating capacities. It is claimed that under this capacity assumption, there is a representative pair (namely the LFDs) such that the Neyman-Pearson test between this pair is minimax optimal. 
Although Huber provides an explicit characterization of the robust hypothesis test in the form of a truncated likelihood ratio, the ``capacities'' condition is required to obtain the optimality result; the LFDs are difficult to obtain in general. Our result is consistent with \cite{huber1965robust} in that our robust test also depends on the least favorable distributions, but we find the LFDs from data by solving a tractable optimization problem.  

More recently, \cite{Levy2009} and \cite{GulZoubir2017} consider uncertainty sets induced by Kullback-Leibler (KL) divergence in the one-dimensional setting without specifying parametric forms; the optimal test is obtained using the strong duality of problem induced by the KL divergence. Aiming to develop a computationally efficient procedure, \cite{nemirovski2015,guigues2016change} consider a convex optimization framework for hypothesis testing, assuming parametric forms for the distributions and the parameters under the null and the alternative hypothesis belong to convex sets. We consider a new way to construct uncertainty sets using Wasserstein metrics and empirical distributions to achieve distributional robustness. Using Wasserstein metric to achieve robustness is a popular technique and has been applied to many areas, including computer vision \cite{rubner2000earth,levina2001earth,rabin2011wasserstein}, generative adversarial networks \cite{arjovsky2017wasserstein,gulrajani2017improved}, and two-sample test \cite{ramdas2017wasserstein}.
    
Our work is also closely related to the Wasserstein distributionally robust optimization (DRO) \cite{esfahani2015data,blanchet2019quantifying,gao2016distributionally,sinha2017certifiable,shafieezadeh2015distributionally}. However, existing DRO problems typically involve only one class of empirical samples, but our problem involves two classes. Hence we cannot rely on existing strong duality results in DRO \cite{blanchet2019quantifying,esfahani2015data,gao2016distributionally} to obtain our results. Besides, we provide new insights regarding our solution's structural properties that are different from those that occurred in other DRO problems. 
Similarly, the line of work in DRO which aims to characterize the size of uncertainty set focuses on a single uncertainty set, including asymptotic results in the finite-dimensional parametric case \cite{blanchet2019robust} and infinite-dimensional case  \cite{si2020quantifying}, as well as non-asymptotic bound \cite{esfahani2015data,shafieezadeh2019regularization,gao2020finite}.
We adopt a similar principle as in \cite{blanchet2019robust,si2020quantifying} but develop different analysis for the case of two uncertainty sets.

\subsection{Organization}

The remainder of the paper is organized as follows. Section~\ref{sec:setup} sets up the problem. Section~\ref{sec:convex detector} presents the optimal test. Section~\ref{sec:radius} characterizes the selection of the radii of the uncertainty sets.  Section~\ref{sec:experiment} demonstrates our robust tests' good performance using both synthetic and real data. Finally, Section~\ref{sec:conc} concludes the paper with some discussions. We delegate all proofs to the appendix. 
 
\section{Wasserstein Minimax Test}\label{sec:setup}
    
Let $\Omega \subset \R^d$ be the sample space, where $d$ is the data dimension. Denote $\scrP(\Omega)$ as the set of Borel probability measures on $\Omega$. 
Given $P_1,P_2\in\scrP(\Omega)$, the \emph{simple hypothesis test} decides whether a given test sample $\omega$ is from $P_1$ or $P_2$. 
In many practical situations, $P_1,P_2$ are not exactly known, but instead we have access to $n_1$ and $n_2$ i.i.d. training samples following distributions $P_1$ and $P_2$, respectively. Denote the two sets of training samples as $\hOmega_k=\{\homega_k^1,\ldots,\homega_k^{n_k}\}$, $k=1,2$, and define empirical distributions constructed using training data sets as
\[Q_k = \frac{1}{n_k}\sum_{i=1}^{n_k} \delta_{\homega_k^i}, \ k=1,2.
\]  Here $\delta_\omega$ denotes the Dirac point mass concentrated on $\omega \in \Omega$.

To capture the distributional uncertainty, we consider \emph{composite hypothesis test} of the form:
\[
\begin{split}
H_0 & ~: \quad \omega \sim P_1, \quad P_1 \in \calP_1; \\
H_1 & ~: \quad \omega \sim P_2, \quad P_2 \in \calP_2,
\end{split}
\]
where $\calP_1,\calP_2$ are collections of relevant probability distributions. In particular, we will consider them to be Wasserstein uncertainty sets.
Below we describe our problem setup. 
  
%   with {\it Wasserstein uncertain sets} around empirical distributions, which is a flexible form to incorporate information from data. 
%   We want to find a {\it minimax} optimal test using samples without assuming parametric forms of the distributions $P_1$ and $P_2$. The minimax optimal test minimizes the risk for the worst-case distributions in the sets.  

%   Assume that $\calP_1,\calP_2\subset \scrP(\Omega)$ are families of probability distributions on $\Omega$, associated with hypotheses $H_1$ and $H_2$, which are also known as ``null'' and ``alternative'' hypotheses. Given a new observation (the test sample) $\omega \in \Omega$, we aim to decide between  two hypotheses:

% The problem above is called the {\it composite hypothesis test}, as each hypothesis consists of a set of distributions. When sets are singletons, it is called the {\it simple hypothesis test}. This problem is also related to binary classification problems in machine learning. The goal is to find an optimal test, which is a mapping from the sample space $\Omega$ to a binary decision.   %$T:\Omega\rightarrow \{1,2\}$: for any sample $\omega\in\Omega$, the test accepts Hypothesis $H_{T(\omega)}$.

\subsection{Randomized test}

We consider the set of all {\it randomized tests} defined as follows \cite{Arkadi2019Statistical}.
\begin{definition}[Randomized test]
Given hypotheses $H_0,H_1$, a randomized test is any Borel measurable function $\pi: \Omega \rightarrow [0,1]$ which, for any observation $\omega \in \Omega$, accepts the hypothesis $H_0$ with probability $\pi(\omega)$ and $H_1$ with probability $1-\pi(\omega)$.
\end{definition}
In the randomized test, the decision to accept a hypothesis can be a random selection based on the function $\pi(\omega)$. Thus, the usual deterministic test (e.g., considered in \cite{gao2018robust}) is a special case by setting $\pi(\omega)\in\{0,1\}$ and the randomized test is more general. 
%   Besides its generality, the randomized test is preferred in this work because it can be extended for multiple hypotheses with a low computational cost.

For a simple hypothesis test with hypotheses $P_1$ and $P_2$, we define the {\it risk} of a randomized test $\pi$ as the summation of Type-I and Type-II errors:
\begin{equation}\label{eq:risk}
\Phi(\pi;P_1,P_2):= \mathbb E_{P_1}[1-\pi(\omega)] + \mathbb E_{P_2}[\pi(\omega)].
\end{equation}
%   where $\mathbb E_{P}$ denotes the expectation with respect to the random variable $\omega$ that follows distribution $P$. 
Here we consider equal weights on the Type-I and Type-II errors; other weighted combinations can be addressed similarly.

\subsection{Wasserstein minimax formulation} 

The minimax hypothesis test finds the optimal test that minimizes the {\it worst-case risk} over all possible distributions in the composite hypotheses: 
\[\label{problem:HT_detector}
\inf_{\pi} \sup_{P_1\in\calP_1,P_2\in\calP_2}  \Phi(\pi;P_1,P_2).
\]
%   where $\calP_1,\calP_2$ are Wasserstein uncertainty sets defined in (\ref{problem:wass vs wass}). 
The resulting worst-case solution $P_1^*,P_2^*$ are called the \emph{least favorable distributions} (LFDs) in the classical robust hypothesis test literature \cite{huber1965robust,huber1973minimax}.

In this paper, we consider uncertainty sets based on the \emph{Wasserstein metric}, defined as:
\[
\wass(P,Q) := \min_{\gamma\in\Gamma(P,Q)} \left\{\E_{(\omega,\omega')\sim\gamma} \left[c(\omega,\omega')\right]  \right\},
\]  
where $c(\cdot,\cdot): \Omega \times \Omega \rightarrow \mathbb R_+$ is a metric on $\Omega$, and $\Gamma(P,Q)$ is the collection of all Borel probability measures on $\Omega\times\Omega$ with marginal distributions $P$ and $Q$. 
% Here we consider Wasserstein metric of order 1 for the ease of exposition; the result can be generalized to other Wasserstein metrics. 
Define the Wasserstein uncertainty sets $\calP_1,\calP_2$ as Wasserstein balls centering at two empirical distributions:
\begin{equation}\label{problem:wass vs wass}
\calP_k := \{P_k\in\scrP(\Omega): \ \wass(P_k,Q_k) \leq \theta_k\}, \quad k = 1,2,
\end{equation}
where $\theta_1,\theta_2>0$ specify the radii of the uncertainty sets. 
%$\theta_1=\theta_2=0$, the uncertainty sets $\calP_1,\calP_2$ reduce to singletons and the problem reduces to a simple hypothesis test. 

\subsection{Comparison with Huber's censored likelihood ratio test} 

Huber's seminal work
\cite{huber1965robust} considered a deterministic minimax test with uncertainty sets referred to as $\epsilon$-contamination sets: 
\[
\calP_k = \{ (1-\epsilon_k)p_k + \epsilon_k f_k, \ f_k \in \scrP(\Omega)\},
\] 
where $\epsilon_k\in(0,1)$, $p_k$ is the nominal density function, and $f_k$ is the density that can be viewed as the perturbation, $k= 1,2$. Huber proved that the optimal test in this setting is a censored version of the likelihood ratio test, with censoring thresholds $c',c''$, and the LFDs are given by:
\[
\begin{aligned}
& q_1(x) = \begin{cases} (1-\epsilon_1)p_1(x) & p_2(x)/p_1(x) < c'' \\
\frac1{c''}(1-\epsilon_1)p_2(x) & p_2(x)/p_1(x) \geq c'' 
\end{cases}; \\  
& q_2(x) = \begin{cases} (1-\epsilon_2)p_2(x) & p_2(x)/p_1(x) > c' \\
c'(1-\epsilon_2)p_1(x) & p_2(x)/p_1(x) \leq c' 
\end{cases}.
\end{aligned}
\]
Huber assumed the exact knowledge of the nominal distributions $p_1$ and $p_2$. This is different from our setting, where we only have limited samples from each hypothesis. A simple observation is that if we set $p_k$ to be the empirical distribution, then the ratio $p_2(x)/p_1(x)$ will be $\infty$ on $\hOmega_2\setminus \hOmega_1$ and $0$ on $\hOmega_1\setminus \hOmega_2$. In such a case, the LFDs proposed by Huber are degenerate 
\[
q_1(x) = \begin{cases} (1-\epsilon_1)/n_1 & x \in \hOmega_1 \\
\epsilon_1/n_2 & x \in \hOmega_2  
\end{cases}; \quad
q_2(x) = \begin{cases} (1-\epsilon_2)/n_2 & x \in \hOmega_2 \\
\epsilon_2/n_1 & x \in \hOmega_1
\end{cases},
\] 
which do not lead to any meaningful test.

\section{Tractable Convex Reformulation and Optimal Test}\label{sec:convex detector}

The saddle point problem  \eqref{problem:HT_detector} for the Wasserstein minimax test is an infinite-dimensional variational problem, which in the original form does not amend to any tractable solution. In this section, we derive a finite-dimensional convex reformulation for finding the optimal test.

We will show the following strong duality result, which means we can exchange the order of infimum and supremum in our problem:
\begin{equation}\label{eq:duality}    
\inf_{\pi} \sup_{P_1\in\calP_1,P_2\in\calP_2}  \Phi(\pi;P_1,P_2) = \sup_{P_1\in\calP_1,P_2\in\calP_2} \inf_{\pi} \Phi(\pi;P_1,P_2).
\end{equation}
This is essential in leading to closed-form expression for the optimal test and convex reformulation in solving the LFDs. 
Our proof strategy is as follows. 
First, in Section \ref{sec:optimal_detector}, we derive a closed-form expression of the optimal test for the simple hypothesis problem $ \inf_{\pi} \Phi(\pi;P_1,P_2)$.
Next in Section \ref{sec:reformulation}, we develop a convex reformulation of the $\sup\inf$ problem on the right-hand side of \eqref{eq:duality}, whose optimal solution gives the LFDs that are supported on the empirical data points. A byproduct of our analysis specifies the optimal test on empirical data points.
Finally, in Section \ref{sec:optimal_test}, we construct the optimal minimax test for the original formulation (left-hand side of \eqref{eq:duality}).
At the core of our analysis is proving that the optimal test can be found by extending the optimal test on the empirical data points to the entire space.

Note that here we cannot directly rely on existing tools such as Sion's minimax theorem \cite{sion1958general}, because (i) the space of all randomized tests is not endowed with a linear topological structure and, (ii) Wasserstein ball is not compact in the space $\scrP(\Omega)$ since $\Omega$ may not be compact. 
      
%   We first present the strong duality result in  Section~\ref{sec:duality} by showing the equivalence between \eqref{problem:HT_detector} and \eqref{problem:supinf}. Therefore, we start with the $\sup\inf$ problem 
%   \begin{equation}\label{problem:supinf}
%     \sup_{P_1\in\calP_1,P_2\in\calP_2} \inf_{\pi} \Phi(\pi;P_1,P_2)
%   \end{equation}
%   and exchange the $\inf$ and $\sup$ in \eqref{problem:HT_detector}.

%   Based on this, we first find the closed-form optimal test for any pairs of distributions $P_1$ and $P_2$ (Section~\ref{sec:optimal_detector}). Then we reduce the problem of finding LFDs $P_1^\ast, P_2^\ast$ from an infinite-dimensional variational problem to a finite-dimensional convex optimization problem, which can be solved efficiently.  
%Finally we construct the optimal test $\pi^*$ based on LFDs $P_1^\ast, P_2^\ast$ (Section \ref{sec:optimal_test}). 

% \subsection{Strong duality}\label{sec:duality}
%   In this subsection, we prove the equivalence between  \eqref{problem:HT_detector} and \eqref{problem:supinf} -- the strong duality, which is critical result to enable our tractable reformulation to find the optimal test and then the LFDs:
%   \begin{proposition}[Strong duality]\label{lemma:saddle}
%     \[\label{eq:duality}    
%     \inf_{\pi} \sup_{P_1\in\calP_1,P_2\in\calP_2}  \Phi(\pi;P_1,P_2) = \sup_{P_1\in\calP_1,P_2\in\calP_2} \inf_{\pi} \Phi(\pi;P_1,P_2).
%     \]
%   \end{proposition}
%   The proof is provided in Appendix \ref{proof:saddle}. 

\subsection{Optimal test for simple hypothesis test}\label{sec:optimal_detector}

Let us start by considering the simple hypothesis test for given $P_1,P_2\in\scrP(\Omega)$, the inner minimization in the right-hand side of (\ref{eq:duality}):
\begin{equation}\label{eq:inf_phi}
\inf_{\pi}\; \Phi(\pi;P_1,P_2).
\end{equation}
Define the total variation distance between two distributions $P_1$ and $P_2$ as
$\TV(P_1,P_2):= (1/2) \int_\Omega |dP_1(\omega) - dP_2(\omega)|.$ 
The following Lemma gives a closed-form expression for the optimal test, which resembles a randomized version of the Neyman-Pearson Lemma. The proof is provided in Appendix \ref{proof:inf_phi}.
% Using strong duality in Theorem~\ref{lemma:saddle}, to solve \eqref{problem:HT_detector}, we can first find the optimal test for a given pair of distributions $P_1$ and $P_2$. The optimal test is characterized by the following Proposition.
%
\begin{lemma}\label{lemma:inf_phi}
Let $p_1(\omega):=\frac{dP_1}{d(P_1+P_2)}(\omega)$.
The test 
\[
  \pi(\omega) = \begin{cases} 1, & \textnormal{if } p_1(\omega) > 1/2, \\
  0, & \textnormal{if }p_1(\omega) < 1/2,\\
  \textnormal{any real number in }[0,1], & \textnormal{otherwise},
  \end{cases}
\]
is optimal for \eqref{eq:inf_phi} with the risk:
\begin{equation}\label{eq:opt_test}
  \psi(P_1,P_2) :=\int_{\Omega}\min\{p_1(\omega),1-p_1(\omega)\}\, d(P_1+P_2)(\omega) = 1-\TV(P_1, P_2).
\end{equation}
\end{lemma}
Lemma \ref{lemma:inf_phi} shows that the optimal test for the simple hypothesis takes a similar form as the likelihood ratio test that accepts the hypothesis with a higher likelihood and breaks the tie arbitrarily. An important observation from the lemma is that the risk only depends on the {\it common} support of the two distributions, defined as $\Omega_0(P_1,P_2) := \big\{\omega\in\Omega: \textstyle 0<p_1(\omega) < 1\big\}$, on which $P_1$ and $P_2$ are absolutely continuous with respect to each other. In particular, if the supports of $P_1,P_2$ have measure-zero overlap, then $\inf_\pi \Phi(\pi;P_1,P_2)$ equals to zero --- the optimal test for two non-overlapping distributions $P_1,P_2$ has zero risk.

%\begin{theorem}[Generalization bound]
%For the optimal test $\pi^\ast$, we have the generalization bound for the true underlying distribution:
%\end{theorem}
%Denote $A_k := \{ \omega, \pi_k(\omega)=0 \}$ which is the rejection region for hypothesis $H_k$. Given $n$ observations $\omega_1,\ldots,\omega_n \stackrel{\text{i.i.d}}{\sim} P_k$, denote their distances to the margin of $A_k$ as $d_1,\ldots,d_n$. More specifically, let
%\[
%d_i = D_A(\omega_i) := \min_{\omega \in A} c(\omega_i,\omega).
%\]
%Then the probability distribution function of $d$ is
%\[
%F(d) =  \int_{D_A(\omega)  \leq d} d(\mathbb P_k(\omega)).
%\]
%Denote the ordered statistic is $d_{(1)} \leq d_{(2)} \leq \ldots \leq d_{(n)}$, then we have the density function of the $m$-th order statistic $d_{(m)}$ equals to 
%\[
%f_{(m)}(d) = nf(d) {{n-1}\choose{m-1}} F(d)^{m-1} ( 1-F(x))^{n-m}.
%\]
%Let $m =\lfloor \sqrt{n} \rfloor$, the the cost of moving $m$ empirical points $\theta(n) \propto \frac{\mathbb E(D_A^{(\lfloor \sqrt{n} \rfloor)}}{\sqrt{n}}$ and the improvement of the objective function is at least $\frac{1}{\sqrt{n}}$.
%Then the mean equals to 

%\begin{remark} Note that the Lipschitz assumption is satisfied in many cases. For example, the Gaussian and Poisson distributions.
%\end{remark}
%

\subsection{Least favorable distributions}\label{sec:reformulation}
 
Now we continue with finding the LFDs given the form of the optimal test in Lemma \ref{lemma:inf_phi}, which corresponds to the remaining supermum part of the right-hand side of (\ref{eq:duality}):
\begin{equation}\label{problem:sup}
%\int_{\Omega}\Big[ \min\big\{\textstyle\frac{dP_1}{d(P_1+P_2)}(\omega),\textstyle\frac{dP_2}{d(P_1+P_2)}(\omega)\big\} \Big]d(P_1+P_2).
% \sup_{P_1\in\calP_1,P_2\in\calP_2} \inf_{\pi} \Phi(\pi;P_1,P_2) = 
\sup_{P_1\in\calP_1,P_2\in\calP_2} \psi(P_1,P_2). 
%\quad \mbox{(LFD problem)}
\end{equation}
Note that from the definition of $\psi$ in (\ref{eq:opt_test}), the risk associated with the optimal test, the problem of finding LFDs admits a clear statistical interpretation: the LFDs correspond to a pair of distributions in the uncertainty sets that are closest to each other in the total variation distance. 

To tackle the infinite-dimensional variational problem \eqref{problem:sup}, let us first discuss some structural properties of the LFDs that will lead to a finite-dimensional convex reformulation. Consider a toy example where $Q_1 = \delta_{\homega_1}, Q_2 = \delta_{\homega_2}$, i.e., there is only one sample in each training data set. The goal of solving LFDs can be understood as moving part of the probability mass on $\homega_1$ and $\homega_2$ to other places such that the objective function $\psi(P_1,P_2)$ is maximized. 
Note that, to find the LFDs, we need to 
(i) move the probability mass such that $P_1$ and $P_2$ overlap as much as possible, since the objective value $\psi(P_1,P_2)$ depends only on the common support; (ii) then if we were to move $p_k$ from $\homega_k$ to a common point $\omega\in\Omega$, $k=1,2$, in the least favorable way, then we solve $\min_{\omega\in\Omega} [p_1c(\omega,\homega_1) + p_2 c(\omega,\homega_2)]$ by the definition of the Wasserstein metric. 
From the triangle inequality satisfied by the metric $c(\cdot,\cdot)$, we need $\omega$ to be on the linear segment connecting $\homega_1$ and $\homega_2$ and in fact, it has to be one of the endpoints $\homega_1$ or $\homega_2$. More generally, one can generalize this argument, and there exist LFDs supported on the empirical observations.

The following lemma shows that the LFDs can be solved via a finite-dimensional convex optimization problem. 
The proof is provided in Appendix \ref{proof:wass vs wass}. For simplicity, define the total number of observations  $n:=n_1+n_2$ and the union of observations from both hypotheses 
\[\hOmega:=\hOmega_1 \cup \hOmega_2.\] Without causing confusions, we re-label the samples in $\hOmega$ as $\{\homega^1,\ldots,\homega^n\}$. 
\begin{lemma}[LFDs]\label{lemma:wass vs wass}
The LFD problem in \eqref{problem:sup} can be reformulated as the following finite-dimensional convex program
 \begin{equation}\label{eq:wass vs wass:dual}
  \begin{aligned}
    &&\max_{\substack{p_1,p_2\in\R_+^{n}\\\gamma_1,\gamma_2\in\R_+^{n\times n}}}\!         &&& \!
     \sum_{l=1}^{n} \min\big\{\textstyle p_1^l, p_2^l \displaystyle\big\}  \\
    && \mbox{{\rm subject to\; }} &&& \sum_{l=1}^{n}\sum_{m=1}^{n} \gamma_{k, l, m} c(\homega^l,\homega^m)\leq \theta_k,\; k=1,2; \\
    &&&&& \sum_{m=1}^{n}\gamma_{k, l, m}= Q_{k}^l,\ 1\leq l \leq n, k = 1,2;\\
            && &&& \sum_{l=1}^{n}\gamma_{k, l, m}=p_k^m, \ 1\leq m\leq n,\;k=1,2.
  \end{aligned}
\end{equation}
\end{lemma}
Above, the decision variables $\gamma_k$ are square matrices that can be viewed as a joint distribution on $\hOmega\times\hOmega$ with marginals specified by $Q_k$ and candidate LFDs $p_k$. The $lm$-th entry of $\gamma_k$ is specified by $\gamma_{k, l, m}$ and the $l$-th entry of $p_k$ (respectively, $Q_k$) is specified by $p_k^l$ (respectively, $Q_k^l$). In the following, we will denote $(P_1^\ast,P_2^\ast)$ as the LFDs solved from \eqref{eq:wass vs wass:dual}. Note that Lemma~\ref{lemma:wass vs wass} simplifies the LFD problem \eqref{problem:sup} from infinite-dimensional to finite-dimensional, using the fact that there exist LFDs supported on a finite set $\hOmega \subset \Omega$ due to our analysis. We also comment that the complexity of solving the LFDs in \eqref{eq:wass vs wass:dual} is \textit{independent} of the dimension of the data, once the pairwise distances $c(\homega^l,\homega^m)$ are calculated and given as input parameters of the convex program. 

%We comment on the complexity of solving the LFDs in \eqref{eq:wass vs wass:dual} regarding the dimension and sample sizes:
% (i) the complexity is \textit{independent of the dimension} since we only need to compute pairwise distances $c(\omega^l,\omega^m)$ as an input to the convex problem; (ii) the complexity regarding the sample sizes can be \textit{nearly sample size independent} with the complexity $O(\ln(n_1)+\ln(n_2))$, since the objective function is Lipschitz in $\ell_1$ norm, according to the complexity theory of the first order method \citep{ben2001lectures}. 

\subsection{Robust optimal test: extension from test on training samples}\label{sec:optimal_test} 

Thus far, we have found one of the LFDs defined on the discrete set of training samples $\hOmega$ by solving the right-hand side of (\ref{eq:duality}), which in turn, defines the optimal test on training samples. However, it may be common in practice that the given test sample is different from all training samples. In this case, the current optimal test in Lemma \ref{lemma:inf_phi} associated with the LFDs is not well-defined on test samples. 
Besides, this optimal test is not uniquely defined when there is a tie between the likelihood of samples under two hypotheses. In this subsection, we will establish an optimal test that is well-defined anywhere in the observation space $\Omega$.

%Note that we have successfully solved the dual form of the original minimax problem. To this end, using Lemma \ref{lemma:wass vs wass}, we first compute the LFDs $(P_1^\ast, P_2^\ast)$ and then define the value of the optimal test $\pi^\ast$ on $\hOmega$ by apply Lemma \ref{lemma:inf_phi} to $(P_1^\ast, P_2^\ast)$.

%   Intuitively, when there is a tie, there are infinitely many optimal tests induced by Lemma~\ref{lemma:inf_phi} if we break the tie arbitrarily. Our goal is to identify the ``true" saddle point among these infinite choices. 
Our main result is the following theorem which specifies the general form of the robust optimal test $\pi^*$ and LFDs $(P_1^\ast,P_2^\ast)$ to the saddle point problem \eqref{problem:HT_detector}, whose proof is given in Appendix \ref{proof:opt_test}.
%to Then $(\pi^\ast;P_1^\ast,P_2^\ast)$ is an optimal solution pair to the saddle point problem \eqref{eq:duality}, thus $\pi^\ast$ is an optimal solution to \eqref{problem:HT_detector}.
%
\begin{thm}[Robust optimal test]\label{thm:opt_test} 
Let $(P_1^\ast,P_2^\ast)$ be the LFDs solved from \eqref{eq:wass vs wass:dual}.
% there exists an optimal Wasserstein minimax test $\pi^*$ such that it is minimax optimal for the problem \eqref{problem:HT_detector} restricted to empirical support $\hOmega$:
% \[
% \sup_{P_1\in \bcalP_1, P_2\in \bcalP_2}  \Phi(\pi^*;P_1,P_2)= \inf_{\pi\in\R^n}\sup_{P_1\in \bcalP_1, P_2\in \bcalP_2}  \Phi(\pi;P_1,P_2).
% \]
The robust optimal test $\pi^*:\Omega\to[0,1]$ to problem \eqref{problem:wass vs wass} is given by
% to problem \eqref{problem:HT_detector} can be   constructed as follows:
\begin{itemize}
\item[(i)] On the support of training samples $\omega\in\hOmega$,  $\pi^\ast(\omega)=\hpi^\ast_m$, for $\omega = \homega^m$, where $\hpi^\ast_m \in[0,1]$, $m = 1, \ldots, n$, is the solution to the following system of linear equations %with respect to $\hpi\in[0,1]^n$: 
\begin{equation}\label{eq:opt_test_constr}
  \begin{aligned}
  \sum_{m=1}^n (1-\hpi_m) P_1^\ast(\homega^m) &= \min_{\lambda_1\geq0} \bigg\{\lambda_1 \theta_1 +  \frac{1}{n_1} \sum_{l=1}^n \max_{1\leq m\leq n} \{1-\hpi_m - \lambda_1 c(\homega^l,\homega^m)\} \bigg\}, 
  \\
  \sum_{m=1}^n \hpi_m P_2^\ast(\homega^m) &=  \min_{\lambda_2\geq0} \bigg\{\lambda_2 \theta_2 +  \frac{1}{n_2} \sum_{l=1}^n \max_{1\leq m\leq n} \{\hpi_m - \lambda_2 c(\homega^l,\homega^m)\} \bigg\};
  \end{aligned}
\end{equation}
the solution is guaranteed to exist.

\item[(ii)] Off the support of training samples $\omega \in \Omega\setminus \hOmega$,   $\pi^*(\omega) \in [\ell(\omega), u(\omega)]$, where 
\begin{equation}\label{eq:pi_bound}
\begin{aligned}
\ell(\omega) &= \max\left\{\max_{i = 1,\ldots,n_1} \min_{\homega\in\hOmega}\left\{\pi^*(\homega) + \lambda_1^* c(\homega,\homega_1^i)  - \lambda_1^* c(\omega,\homega_1^i)   \right\}, 0 \right\}, %\label{eq:pi_bound}
\\
u(\omega) &= \min\left\{ \min_{j = 1,\ldots,n_2} \max_{\homega\in\hOmega}\{ \pi^*(\homega_2^j) - \lambda_2^* c(\homega,\homega_2^j)+ \lambda_2^* c(\omega,\homega_2^j)\},1\right\}, %\label{eq:pi_upper}
\end{aligned}
\end{equation} 
$\lambda_k^*$, $k=1,2$ are the minimizers to the $\inf$ problems on the right hand side of \eqref{eq:opt_test_constr}, and it is guaranteed that $u(\omega)\geq \ell(\omega)$, $\forall \omega \in \Omega\setminus \hOmega$.\\ 
% then
% \[
% \sup_{P_1\in\calP_1,P_2\in\calP_2} \Phi(\pi^*,P_1,P_2) = \sup_{P_1\in \bcalP_1, P_2\in \bcalP_2} \Phi(\pi^*,P_1,P_2).
% \]
\end{itemize}
\end{thm}

% We comment that the optimal test value $\pi^*(\homega)$, for any empirical point $\homega$, can be found by checking the optimality condition in \eqref{eq:opt_test_constr}, which is a collection of linear constraints for $\pi^*\in\R^{n}$ and the choice of $\pi^*$ may not be unique. 
The first part of the theorem defines the optimal test on training samples, resulting from the finite-dimensional saddle point problem 
\[
  \sup_{P_1\in\bcalP_1,P_2\in\bcalP_2} \inf_{\pi:\hOmega\to[0,1]} \Phi(\pi;P_1,P_2),
\]
where $\bcalP_k := \calP_k \cap \scrP(\hOmega)$, $k= 1, 2$.
By Lemma \ref{lemma:wass vs wass}, this is equivalent to the right-hand side of \eqref{eq:duality}.
The second part extends the optimal test on training samples to the whole space. This is a non-trivial results that build on the properties of Wasserstein metric and the duality result.

To illustrate Theorem \ref{thm:opt_test}, let us consider a toy example as shown in Figure~\ref{fig:opt_test_illustration}. Suppose the training samples for hypothesis $H_0$ is $\homega_1=-2$ and for hypothesis $H_1$ are $\homega_2=1$ and $\homega_3 = 3$. Then, the two empirical distributions $Q_1$ is a point mass on $\homega_1=-2$ and $Q_2$ is a discrete distribution that $\homega_2=1$ and $\homega_3 = 3$ occur with equal probability 1/2. By setting the radii of the uncertainty sets $\theta_1=\theta_2=1$, the LFDs solution to \eqref{eq:wass vs wass:dual} becomes $P_1^*(\homega_1)=0.69$, $P_1^*(\homega_2)=0.28$, $P_1^*(\homega_3)=0.03$, and $P_2^*(\homega_1)=0.29$, $P_2^*(\homega_2) = 0.28$, $P_2^*(\homega_3)=0.43$. Notice that there is a tie at the point $\homega_2$. Now we will invoke Theorem \ref{thm:opt_test} to break this tie. According to \eqref{eq:opt_test_constr}, the robust optimal test $\pi^*(\homega_i)$, $i= 1, 2, 3$ needs to satisfy
\[
1-\pi^*(\homega_1)-\lambda_1^* c(\homega_1,\homega_1) = 1-\pi^*(\homega_2)-\lambda_1^* c(\homega_1,\homega_2) = 1-\pi^*(\homega_3)-\lambda_1^* c(\homega_1,\homega_3). 
\]
Therefore, we can set $\pi^*(\homega_2) =1-c(\homega_1,\homega_2)/c(\homega_1,\homega_3)=0.4$. This means that the optimal test at $\homega_2$ should accept the hypothesis $H_0$ with probability $0.4$ (note that the tie is not broken arbitrarily). 
As a comparison, consider a different case where $\homega_2 = 2$ while everything else is kept the same. It can be verified that there is still a tie at $\homega_2$. However, this time we have $\pi^*(\homega_2) = 1-c(\homega_1,\homega_2)/c(\homega_1,\homega_3)=0.2$, meaning that the optimal test at $\homega_2$ should accept the hypothesis $H_0$ with probability $0.2$. We note that in this simple experiment, the chance of accepting $H_0$ decreases if we move $\homega_2$ away from $\homega_1$, which is consistent with our intuition as illustrated in Figure~\ref{fig:opt_test_illustration}. Moreover, we also plot the upper and lower bounds $u(\omega)$ and $\ell(\omega)$, as defined in \eqref{eq:pi_bound}, showing the range of the optimal test off the support of training samples. This example also demonstrates the advantage of using Wasserstein metrics in defining the uncertainty sets: the optimal test will directly reflect the data geometry.

\begin{figure}[!ht]
\begin{center}
\includegraphics[width=1\textwidth]{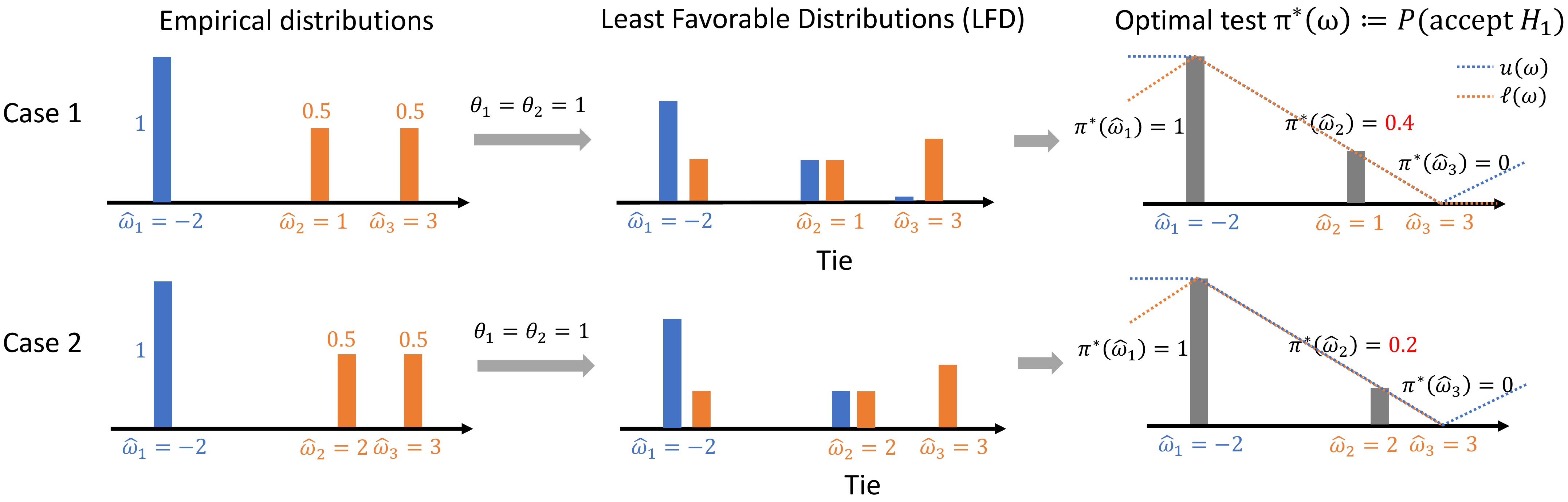}
\captionof{figure}{A toy example illustrating the optimal test depends on the training data configuration. In these two cases, there are three samples, and only $\homega_2$ is different, which takes values 1 and 2, respectively. Note that the optimal test $\pi^*(\homega_2)$ will change when the gap between empirical samples are different. We also illustrate the upper and lower bounds $u(\omega)$ and $\ell(\omega)$ from (\ref{eq:pi_bound}).}
\label{fig:opt_test_illustration}
\end{center}
\end{figure}

% \subsection{Extended optimal test}\label{sec:principle}
% As a result of \eqref{eq:opt_test_constr}, we obtain an optimal test $\pi^*$ supported on $\hOmega$. To make it generally applicable, we also extend the test $\pi^*$ to the whole space $\Omega$ in Theorem~\ref{thm:opt_test} such that the corresponding worst-case risk does not increase. In other words, we find $\pi^*:\Omega \rightarrow [0,1]$ such that
% \begin{equation}\label{pie}
% \sup_{P_1\in\calP_1,P_2\in\calP_2} \Phi(\pi^*;P_1,P_2)  = \sup_{P_1\in \bcalP_1, P_2\in \bcalP_2} \Phi(\pi^\ast;P_1,P_2).
% \end{equation}

% We note that for completeness of the argument, we still need to characterize the relation between the upper bound for $\pi^*(\omega)$ in \eqref{eq:pi_upper} and the lower bound in \eqref{eq:pi_bound}, to check whether they are consistent. A conjecture is that the consistency holds almost surely in the sample space, and the detailed proof is left for future work. 

\subsection{Extension to whole space via kernel smoothing}
\label{sec:sub_kernel}
We observe that for samples $\omega$ off the empirical support, it is possible to have $u(\omega)$ strictly larger than $\ell(\omega)$ with $u(\cdot),\ell(\cdot)$ given in Equation \eqref{eq:pi_bound}. In such cases, there are infinite choices for $\pi^*(\omega)$ according to Theorem~\ref{thm:opt_test}. In this subsection, we describe a specific choice for $\pi^*(\omega)$ under such situation by kernel smoothing. As a natural strategy, we may use kernel smoothing to extend LFDs solved from \eqref{eq:wass vs wass:dual} to the whole space. This can be done by convolving the discrete LFDs with a kernel function $G_h:\mathbb R^d \rightarrow \mathbb R$ parameterized by a (bandwidth) parameter $h$: 
\begin{equation}
P_k^h(\omega):= \sum_{l=1}^nP_k^\ast(\homega^l) G_h( \omega - \homega^l), \ k = 1,2, \ \forall \omega \in \Omega.
\label{smoothed_LFDs}
\end{equation}
There can be various choices of kernel functions. For instance, given normalized data, we can use the product of one-dimensional kernel function $g: \mathbb R \rightarrow \mathbb R$ with bandwidth $h > 0$:
\[
G_h(x) =  \frac{1}{h^d}\prod_{i=1}^d  g\left(\frac{x_i}{h}\right), x \in \mathbb R^d.
\] 
An example of the kernel-smoothed LFDs is shown in Figure~\ref{fig:LFD_illustration}. Through convolution, we can obtain the kernel-smoothed LFDs and the corresponding test $\pi^\ast_h$ that is defined as the optimal test for the simple hypothesis under $(P_1^h,P_2^h)$ as specified in Lemma~\ref{lemma:inf_phi}. To ensure the risk after kernel-smoothing is comparable to that of the robust optimal test $\pi^*$, we truncate the resulted $\pi^\ast_h$ such that \eqref{eq:pi_bound} is satisfied after truncation. After such a procedure, the test based on the kernel-smoothed LFDs will achieve a good performance as validated by the numerical experiments in Section \ref{sec:experiment}.

\subsection{Test with batch samples}\label{sec:multi-obs}
Testing using a batch of samples is important in practice, as one test sample may not achieve sufficient power. We can construct a test for a batch of samples by assembling the optimal test for each individual sample. Assume $m$ i.i.d. test samples $\omega_1$, $\omega_2$, $\ldots$, $\omega_m$. Consider a {\it batch test} based on the ``majority rule'' with the acceptance region for $H_0$  defined as
$\mathbb A:=\{(\omega_1,\omega_2,\ldots,\omega_m):\pi^m(\omega_1,\omega_2,\ldots,\omega_m) \geq 1/2\},$
where
\[
\pi^m(\omega_1,\omega_2,\ldots,\omega_m) = \frac1m\sum_{i=1}^m \pi^\ast(\omega_i),
\]
can be viewed as the fraction of votes in favor of hypothesis $H_0$ (due to Lemma~\ref{lemma:inf_phi}). 
We can bound the risk of such a majority rule batch test:
\begin{proposition}[Risk for batch test]\label{multi_obs_risk}
The risk of the test $\pi^m(\omega_1,\ldots,\omega_m)$
is be upper bounded by
\[
\begin{aligned}
& \max\left\{ \sup_{P_1\in\calP_1} \mathbb{P}_{P_1} \left[ \mathbb A^c \right], \sup_{P_2\in\calP_2} \mathbb{P}_{P_2} \left[ \mathbb A \right] \right\} 
\leq    \sum_{m/2 \leq i \leq m} {m \choose i} (\epsilon^\ast)^i(1-\epsilon^\ast)^{m-i},
\end{aligned}
\] 
where 
\[\epsilon^\ast = \sup_{P_1\in\calP_1,P_2\in\calP_2}\Phi(\pi^\ast;P_1,P_2),\] is the worst-case risk of the optimal randomized test and $\mathbb A$ is the acceptance region for $H_0$. Thus, when $\epsilon^\ast < 1/2$, the above probability tends to 0 exponentially fast as the batch size $m \rightarrow \infty$.
\end{proposition}

\section{Radii Selection}\label{sec:radius}

In this section, we discuss how to select the radii $\theta_1,\theta_2$, which is critical to the performance of the robust optimal test. There is clearly a trade-off: when the radius is too small, the optimal test is not robust and does not generalize well to new test data; while the radius is too large, the solution may be too conservative, causing performance degradation. We expect sample sizes $n_1$ and $n_2$ to play a major role in determining the radii, and thus in the following we emphasize by denoting the radii as $\theta_{k,n_k}$ and the empirical distributions as $Q_{k,n_k}$. It should also be remembered that the uncertainty sets $\calP_{k}(\theta_{k,n_k})$, $k=1,2$, also depend on the sample sizes. 

To characterize the radii selection, we adopt the profile-based inference proposed by \cite{blanchet2019robust}, which extends the empirical likelihood method for divergence-based distributionally robust optimization \cite{lam2019recovering,duchi2019variance} by replacing likelihood with transport cost.
It selects the radii based on the principle that the distributional uncertainty set should contain a pair of distributions whose resulting optimal test (for the corresponding simple hypothesis test) coincides with the optimal test for the underlying true distributions.   
More precisely, let $P_1^\circ,P_2^\circ$ be the underlying true distributions of the hypotheses $H_0$ and $H_1$ respectively. Define the {\it oracle} test $\pi^\circ$ as the optimal test of the simple hypothesis test associated with $P_1^\circ,P_2^\circ$, which is specified by Lemma~\ref{lemma:inf_phi}.  
Also define the set of optimal tests for resolving simple hypothesis test associated with each pair of distributions in our uncertainty sets (using Lemma \ref{lemma:inf_phi}) as
\[
\Pi(\theta_{1,n_1},\theta_{2,n_2}) := \left\{\pi:  \exists P_1\in\calP_{1}(\theta_{1,n_1}), P_2\in\calP_{2}(\theta_{2,n_2}) \text{ such that } 
\pi \in \argmin_{\pi'} \Phi(\pi'; P_1, P_2) \right\}.
\]
We are interested in finding the radii such that the set is likely to include the oracle test, i.e.,  $\pi^\circ\in\Pi(\theta_{1,n_1},\theta_{2,n_2})$. To achieve this goal, we introduce a set $\mathcal S$ that contains all possible pairs of distributions giving rise to the oracle test $\pi^\circ$: 
\[
\mathcal S := \left\{(P_1, P_2)\in\scrP(\Omega)\times\scrP(\Omega): \pi^\circ \in \argmin_{\pi:\Omega\to[0,1]} \Phi(\pi;P_1,P_2) \right\}.\]
Note that $\mathcal{S}$ is guaranteed to be non-empty since it contains at least the true distribution $\{P_1^\circ, P_2^\circ\}$. Then consider within $\mathcal S$, the distributions that are closest to the empirical distributions $Q_{k,n_k}$ and define the so-called \emph{profile function} to capture the notion of ``distance to the empirical distributions'' within the set:
\begin{equation}\label{eq:profile}
\profile_{n_1,n_2} := \inf_{\{P_1,P_2\} \in \mathcal S}\max_{k=1,2}\; \wass(P_k,Q_{k,n_k}),
\end{equation}
here the subscript indicates its dependence on the sample sizes $n_1$ and $n_2$. Clearly if the radii $\theta_{1,n_1},\theta_{2,n_2} \geq \profile_{n_1,n_2}$, then the intersection $(\calP_{1}(\theta_{1,n_1})\times\calP_{2}(\theta_{2,n_2})) \cap \mathcal{S}$ is nonempty, and thus $\pi^\circ\in\Pi(\theta_{1,n_1},\theta_{2,n_2})$, as illustrated in Figure~\ref{fig:profile}.

% The proof strategy is as follows. We define the optimal test for true distributions as the {\it oracle}; however, since the true distributions are unknown, such an oracle test is also unknown. Our goal is to select the smallest radii that are meanwhile sufficiently large such that the oracle test is achievable with high probability. To this end, we define a set $\mathcal S$ that contains all pairs of distributions that lead to the oracle test. Then within the set $\mathcal S$, we find the closest distribution to the empirical distributions with respect to the Wasserstein metric; intuitively, this gives rise to the tightest radii that meet our goal. 
Our goal is to find an asymptotic upper bound of such distance and use it as the radii; such a choice will be such that the robust optimal test lies in the confidence region of $\pi^\circ$. Indeed, if we can provide a theoretical upper bound for the asymptotic value of the right-hand side of \eqref{eq:profile}, then by setting the radii accordingly, the intersection $(\calP_{1}(\theta_{1,n_1})\times\calP_{2}(\theta_{2,n_2})) \cap \mathcal{S}$ is nonempty and thus $
\pi^\circ\in\Pi(\theta_{1,n_1},\theta_{2,n_2})$.
% \[
% \liminf_{n_1,n_2\to\infty} \P\{\pi^\circ\in\Pi(\theta_{1,n_1},\theta_{2,n_2})\} \geq 1-\alpha.
% \]
From the strong duality in \eqref{eq:duality} which has been proved in the previous section, any optimal solution to the left-hand side of \eqref{eq:duality} will belong to the set $\Pi(\theta_{1,n_1},\theta_{2,n_2})$. This ensures that the optimal test we obtained belongs to the confidence region for the oracle test $\pi^0$.

\begin{figure}[!ht]
\begin{center}
\includegraphics[width = .6\textwidth]{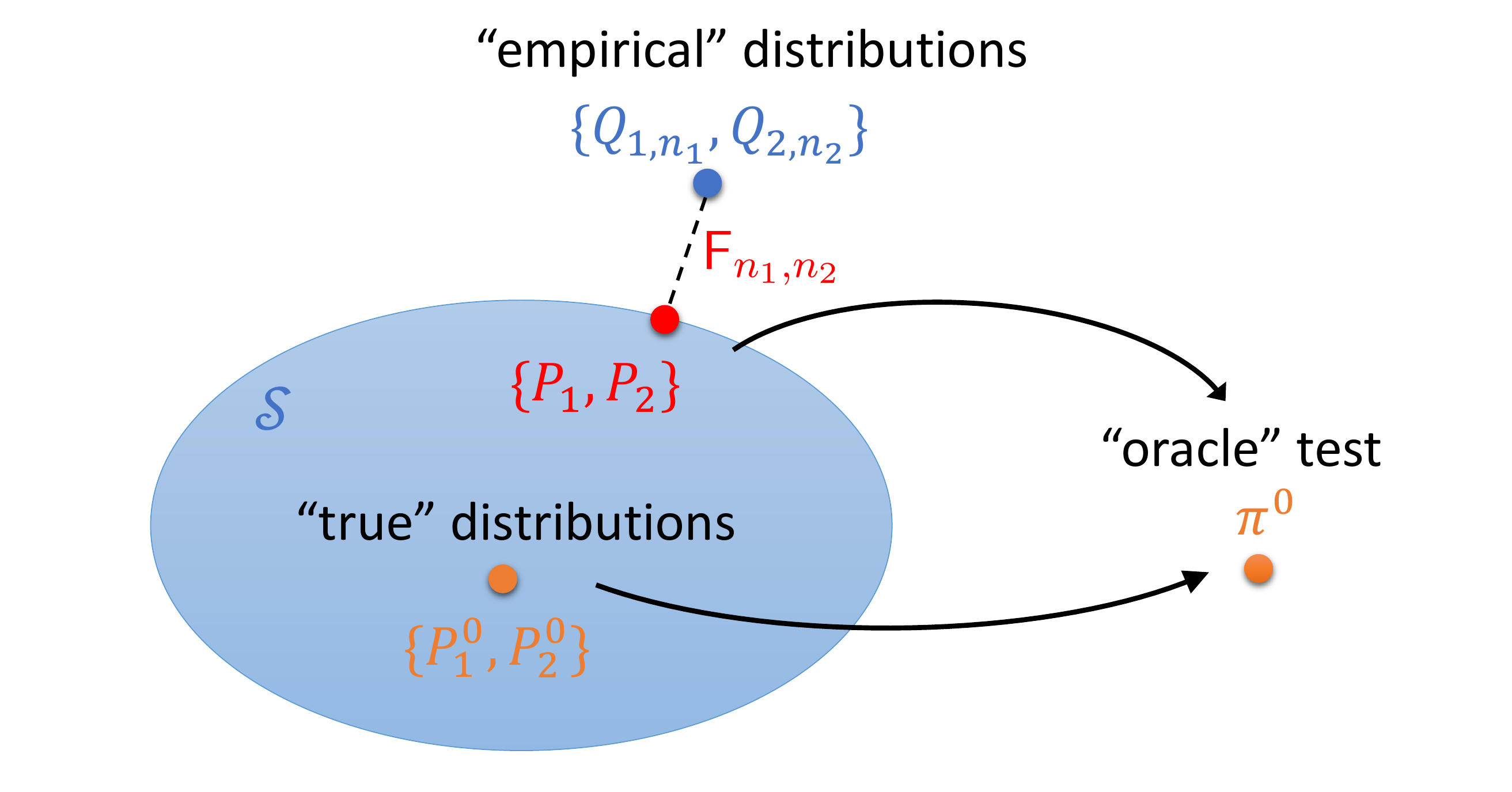} 
\end{center}
\caption{An illustration of the profile function. The set $\mathcal S$ contains all pairs of distributions $\{P_1,P_2\}$ such that the oracle test is optimal; $\profile_{n_1,n_2}$ denotes the minimal distance from the empirical distribution to the set $\mathcal S$.}
\label{fig:profile}
\end{figure}

We first derive an equivalent dual representation of the profile function $\profile_{n_1,n_2}$. We introduce some additional definitions and notations as follows. We partition the sample space $\Omega$ as
\[
\Omega_1^\circ := \left\{ \omega \in \Omega: dP_1^\circ(\omega) \geq dP_2^\circ(\omega)\right\},\quad  \Omega_2^\circ := \left\{ \omega \in \Omega: dP_1^\circ(\omega) < dP_2^\circ(\omega)\right\}.
\]
Thereby the oracle test $\pi^\circ$ accepts hypothesis $H_0$ on set $\Omega_1^\circ$ and accept hypothesis $H_1$ on set $\Omega_2^\circ$. 
The boundary between $\Omega_1^\circ$ and $\Omega_2^\circ$ corresponds to the decision boundary of the oracle test $\pi^\circ$; the boundary is typically of measure zero for continuous distributions. 
Denote by $\cB_+(\Omega)$ ($\Lip(\Omega)$) the set of bounded and non-negative (respectively, 1-Lipschitz continuous) functions on $\Omega$. Define the function class:
\begin{equation}\label{eq:setA}
\cA := \Big\{\bm\alpha=\bm\alpha_2 \ind_{\Omega_2^\circ} - \bm\alpha_1 \ind_{\Omega_1^\circ}:\;\bm\alpha_k\in\cB_+(\Omega_k^\circ)\cap\Lip(\Omega_k^\circ),\,\bm\alpha(\omega_k^\circ)=0,\,k=1,2 \Big\},   
\end{equation}
where $\ind$ is the indicator function and $\omega_k^\circ\in\Omega_k^\circ$, $k=1,2$.
Thus for each function $\bm\alpha\in\cA$, the positive part is on $\Omega_2^\circ$ and the negative part is on $\Omega_1^\circ$, and all functions in $\cA$ coincide on $\omega_1^\circ,\omega_2^\circ$.
We have the following lemma, whose proof is given in Appendix \ref{sec:B.1}.

\begin{lemma}
\label{lemma:profile_dual}
The profile function $\profile_{n_1,n_2}$ defined in \eqref{eq:profile} equals
\[
  \profile_{n_1,n_2} =  \sup_{\substack{\lambda_1,\lambda_2\geq 0,\lambda_1+\lambda_2 \leq 1\\\bm\alpha\in\cA}}
  \begin{multlined}[t] \Bigg\{ \E_{\homega_1\sim Q_{1,n_1}} \Big[\inf_{\omega\in\Omega} \big\{ \lambda_1 c( \omega,\homega_1) + \bm\alpha(\omega) \big\} \Big] \\
  + \E_{\homega_2\sim Q_{2,n_2}} \Big[\inf_{\omega\in\Omega} \big\{ \lambda_2 c( \omega,\homega_2) - \bm\alpha(\omega) \big\}
  \Big] \Bigg\}.
  \end{multlined}
\]
\end{lemma}
The objective function, denoted as $F_{n_1,n_2}(\lambda_1,\lambda_2,\bm\alpha)$, of the above supreme problem can be decoupled into two terms:  
%as follows:
$\label{eq:profile_dual_indicator}
F_{n_1,n_2}(\lambda_1,\lambda_2,\bm\alpha)= E_{n_1,n_2}(\lambda_1,\lambda_2,\bm\alpha) +  G_{n_1,n_2}(\bm\alpha)$,
where 
\[\begin{aligned}
E_{n_1,n_2}(\lambda_1,\lambda_2,\bm\alpha) &:= 
\begin{multlined}[t]
\frac{1}{n_1} \sum_{i=1}^{n_1} \inf_{\omega\in\Omega}\{ \lambda_1c(\omega,\homega_1^i) + \bm\alpha(\omega)-\bm\alpha(\homega_1^i) \} \\
+ \frac{1}{n_2} \sum_{j=1}^{n_2} \inf_{\omega\in\Omega}\{ \lambda_2c(\omega,\homega_2^j) - (\bm\alpha(\omega) - \bm\alpha(\homega_2^j))\},
\end{multlined} \\
G_{n_1,n_2}(\bm\alpha) &:= \frac{1}{n_1} \sum_{i=1}^{n_1(N)}   \bm\alpha(\homega_1^i) - \frac{1}{n_2} \sum_{j=1}^{n_2(N)}   \bm\alpha(\homega_2^j).
\end{aligned}
\]
It follows that $E_{n_1,n_2}(\lambda_1,\lambda_2,\bm\alpha)\leq 0$ since the $\inf$ value is non-positive by taking $\omega=\homega_1^i$ and $\omega=\homega_2^j$, respectively, whence $F_{n_1,n_2}(\lambda_1,\lambda_2,\bm\alpha)\leq G_{n_1,n_2}(\bm\alpha)$ and 
$$\profile_{n_1,n_2} \leq \sup_{\bm\alpha\in\cA} G_{n_1,n_2}(\bm\alpha).$$
Based on the definition of $\cA$ in \eqref{eq:setA}, we observe a close-form solution for $\sup_{\bm\alpha\in\cA} G_{n_1,n_2}(\bm\alpha)$ as follows.
By definition of $\cA$, $\bm\alpha(\homega_1^i)\leq0$ for $\homega_1^i\in\Omega_1^\circ$ and $\bm\alpha(\homega_2^j)\geq0$ for $\homega_2^j\in\Omega_2^\circ$. Therefore, to maximize $G_{n_1,n_2}(\bm\alpha)$, we can set $\bm\alpha(\homega_1^i)=0$ for $\homega_1^i\in\Omega_1^\circ$ and $\bm\alpha(\homega_2^j)=0$ for $\homega_2^j\in\Omega_2^\circ$. 
% Indeed, if for $\homega_1^i\in\Omega_1^\circ$, we let $\bm\alpha(\homega_1^i)=-\Delta$ for any constant $\Delta>0$, then $G_N(\bm\alpha)$ attains its maximum by setting $\bm\alpha(\homega_2^j)=-\Delta-\min_{\homega_1^i \in\Omega_1^\circ}c(\homega_2^j,\homega_1^i)$ for $\homega_2^j \in \Omega_1^\circ$. By the definition of the set $\Omega_1^\circ$, we have for any $\epsilon>0$, there exists $N_0$ such that for any $N>N_0$, we have
% \[
% \P[|\{\homega_1^i: \homega_1^i\in\Omega_1^\circ\}| > |\{\homega_2^j: \homega_2^j\in\Omega_1^\circ\}|] \geq 1-\epsilon,
% \]
% which implies that with high probability the objective value of $G_N(\bm\alpha)$ by setting $\bm\alpha(\homega_1^i)=-\Delta<0$ is smaller than the value obtained by setting $\bm\alpha(\homega_1^i)=0$, and similar conclusion holds for $\bm\alpha(\homega_2^j)$. 
%
% Now let $\bm\alpha(\homega_1^i)=0$ for $\homega_1^i\in\Omega_1^\circ$ and $\bm\alpha(\homega_2^j)=0$ for $\homega_2^j\in\Omega_2^\circ$. 
In addition, since $ \bm\alpha_1,\bm\alpha_2$ are $1$-Lipschitz, we have $\bm\alpha(\homega_1^i) \leq \min_{j:\;\homega_2^j \in\Omega_2^\circ}c(\homega_1^i,\homega_2^j)$ for $\homega_1^i\in\Omega_2^\circ$ and $\bm\alpha(\homega_2^j) \geq -\min_{i:\;\homega_1^i \in\Omega_1^\circ}c(\homega_2^j,\homega_1^i)$ for $\homega_2^j\in\Omega_1^\circ$. 
Hence we have
\begin{equation}\label{eq:bound_dist_text}
\sup_{\bm\alpha\in\cA} G_{n_1,n_2}(\bm\alpha) = \frac{1}{n_1} \sum_{i:\;\homega_1^i\in\Omega_2^\circ}\min_{j:\;\homega_2^j \in\Omega_2^\circ}c(\homega_1^i,\homega_2^j) + \frac{1}{n_2} \sum_{j:\;\homega_2^j\in\Omega_1^\circ} \min_{i:\;\homega_1^i \in\Omega_1^\circ}c(\homega_2^j,\homega_1^i). 
\end{equation}
% \begin{remark}

Note that the profile function defined in \eqref{eq:profile} measures the minimal transport cost from the empirical distributions to some distribution in the set $\mathcal{S}$ that yields the same optimal test as the oracle test.
From this perspective, the right-hand side of \eqref{eq:bound_dist_text} provides an upper bound on such minimal transport cost. 
It basically suggests to move those empirical samples $\homega_1^i$ (resp.~$\homega_2^j$) falling into the wrong region $\Omega_2^\circ$ (resp.~$\Omega_1^\circ$) to the closest empirical samples in a different class $\argmin_{j:\;\homega_2^j \in\Omega_2^\circ}c(\homega_1^i,\homega_2^j)$ (resp.~$\argmin_{i:\;\homega_1^i \in\Omega_1^\circ}c(\homega_2^j,\homega_1^i)$). 
Thereby, this form sheds light on an approximate optimal distributions of \eqref{eq:profile} that are obtained by moving empirical points to some neighboring points in a different class.
The resulting distributions can be different from the true distribution, but yield an optimal test close to the oracle test.
% \end{remark}

% In the remainder of this section, we list some main findings based on the upper bound of the profile function as shown in \eqref{eq:bound_dist_text}. 
% We say $X_n \lesssim_D Y_n$ if $\lim_{n\rightarrow\infty} \E [ f(X_n)] \leq  \E [ f(Y_n)]$ for every continuous and bounded non-decreasing function $f(\cdot)$. 

Next, we compute the asymptotic value of $\sup_{\bm\alpha\in\cA} G_{n_1,n_2}(\bm\alpha)$ using \eqref{eq:bound_dist_text}, 
which only involves the minimum-distance-type statistics of two sets of sample, which are easier to analyze than $\profile_{n_1,n_2}$.
We consider a balanced sample size regime.
% by assuming $n_1 = c_1 N$ and $n_2 = c_2 N$ with fixed constants $c_1,c_2>0$ which can be unequal. Thereby we focus on the asymptotic scenario by letting $N$ goes to infinity. 
%$\lim_{N\to\infty} n_1(N)/n_2(N) = 1$, although our results can be easily generalized to the setting where the sample size ratio converges to a positive constant. 
%
\begin{thm}\label{thm:radius}
Suppose $\lim_{n_1,n_2\to\infty} n_2/n_1= c>0$.
Assume that $f_1$ and $f_2$ are respectively the density functions of $P_1^\circ$ and $P_2^\circ$ that are absolutely continuous to each other and satisfy
\[
\int_{\Omega_1^\circ} f_2(x)f_1(x)^{-1/d}dx < \infty, \ \int_{\Omega_2^\circ} f_1(x)f_2(x)^{-1/d}dx < \infty,
\]
and for some $\epsilon>0$ it holds that
\[\begin{aligned}
& \sup_{n \in \mathbb{N}} \E_{x\sim f_1|_{\Omega_2^\circ}, x_1,\ldots,x_n\sim f_2|_{\Omega_2^\circ}}[(n^{1/d}\min_{1\leq i\leq n}  \left\Vert x - x_i\right\Vert)^{1+\epsilon}] < \infty,\\ 
& \sup_{n \in \mathbb{N}} \E_{x\sim f_2|_{\Omega_1^\circ}, x_1,\ldots,x_n\sim f_1|_{\Omega_1^\circ}}[(n^{1/d}\min_{1\leq i\leq n}  \left\Vert x - x_i\right\Vert)^{1+\epsilon}]<\infty,
\end{aligned}
\]
where $f|_{A}$ denotes the density of the restriction of distribution $f$ on a set $A$. Then 
\begin{equation}\label{eq:upper_bound}
\begin{aligned}
n_1^{1/d}\sup_{\bm\alpha\in\cA} G_{ n_1, n_2}(\bm\alpha) \rightarrow \frac{\Gamma(1+1/d)}{V_d^{1/d}} \left(c^{-1/d}\int_{\Omega_2^\circ} \frac{f_{1}(x)}{[f_{2}(x)]^{1/d}}dx + \int_{\Omega_1^\circ} \frac{f_{2}(x)}{[f_{1}(x)]^{1/d}} dx\right),
\end{aligned}
\end{equation}
in $L^1$ as $n_1,n_2\to\infty$, where $V_d = \pi^{d/2}/\Gamma(1+d/2)$ is the volume of the unit ball in $\R^d$, and $\Gamma(x)=\int_0^\infty z^{x-1}e^{-z}dz$ is the Gamma function.
\end{thm}

The assumptions on the true data-generating densities resemble the assumptions required for computing the nearest neighbor distances in \cite{evans2002asymptotic,penrose2011laws,penrose2003weak}. 
Under these assumptions, the weak law of large numbers is applied to the right-hand side of \eqref{eq:bound_dist_text}.
They can be satisfied under several scenarios, which includes but not limited to: (i) the set $\Omega_1^\circ, \Omega_2^\circ$ are both a finite union of convex bounded sets with non-empty interior and the restricted density $f_1|_{\Omega_1^\circ},f_2|_{\Omega_2^\circ}$ are bounded away from zero, and (ii) the restricted densities satisfy that for some $r> d/(d-1)$, we have $\int_{\Omega_j^\circ} \left\Vert x \right\Vert_2^r f_k|_{\Omega_j^\circ}(x)dx < \infty,k,j=1,2$ \cite{penrose2011laws}.

The first component on the right-hand side of \eqref{eq:upper_bound} equals
the limit of the expectation of $n_2^{1/d} \min_{1\le i\le n_2}\norm{x-x_i}$, where $x\sim f_{1}|_{\Omega_2^\circ}$ and $x_i\sim f_{2}|_{\Omega_2^\circ}$; similar for the second component.
It is computed by a conditioning argument where we condition on the random variable with respect to which we compute its nearest-neighbor distance, following a same argument as in \cite[Lemma~3.2]{penrose2003weak}.
Observe that  $\int_{\Omega_2^\circ} f_1/f_2^{1/d} dx = \int_{\Omega_2^\circ} (f_1/f_2)^{1/d} f_1^{\frac{d-1}{d}} dx $.
Hence it depends on the true densities and the value will be smaller if the density $f_1$ is relatively smaller on the set $\Omega_2^\circ$, and if the density ratio $f_1/f_2$ is close to 0 (note that it is always less than or equal to 1 on $\Omega_2^\circ$). This indicates that our choice of the radii tends to be smaller for distributions that are more different and thus it would be easier to distinguish between them.

Based on our principle, Theorem \ref{thm:radius} shows that our choice of the radii will be of the order $\mathcal O(n_1^{-1/d})$ under a balanced sample size regime. Since our framework yields a non-parametric test, this order is consistent with other non-parametric methods, and represents only the worst-case scenario and may be improved if additional conditions on the true data-generating distributions are imposed.
We would like to point out that although the same order can be obtained using the concentration principle that the uncertainty sets contain true distributions with high probability  \cite{canas2012learning,fournier2015rate,esfahani2015data}, our bound in \eqref{eq:upper_bound} provides a more informative constant term that involves the density ratio of the two underlying distributions; while the constant term obtained from the concentration principle would not involve any relationship between the two underlying distributions. 
% \end{remark}

Moreover, we remark that if the support $\Omega_1^\circ,\Omega_2^\circ$ are compact convex sets and the restricted densities $f_1|_{\Omega_1^\circ},f_2|_{\Omega_2^\circ}$ are continuous, bounded away from zero, and has bounded partial derivatives, then the rate of convergence has been provided explicitly in \cite{evans2002asymptotic}: for all $0< \rho < 1/d$, we have that as $N\rightarrow \infty$, the higher order terms on the right-hand side of \eqref{eq:upper_bound} would be $\mathcal O\left(n_1^{-(1/d-\rho)}\right)$.

We also remark that although we adopt a similar principle as used in \cite{blanchet2019robust,si2020quantifying} by considering the profile function $\profile_{n_1,n_2}$, the proof in our case is much more challenging, because: (1) the uncertainty set here involves the empirical samples from two classes instead of one uncertainty set; (2) the introduced variable ${\bm\alpha}_1,{\bm\alpha}_2$ are functions in the continuous samples space instead of a finite-dimensional vector, thus the optimality condition is not a simple first-order condition but involves inequalities yielding from variational principle, resulting in an additional constraints for solving $\profile_{n_1,n_2}$. Thus, we develop quite different analytical techniques to obtain the results. Details can be found in Appendix~\ref{proof:radius}.

\section{Numerical Experiments}\label{sec:experiment}

In this section, we present several numerical experiments to demonstrate the good performance of our method.

\subsection{Synthetic data: Testing Gaussian mixtures}

Assume the dimension is $100$ and the samples under two hypotheses are generated from Gaussian mixture models (GMM) following the distributions $0.5\mathcal N(0.4{e},I_{100}) + 0.5\mathcal N(-0.4{e},I_{100})$ and $0.5\mathcal N(0.4{f},I_{100}) + 0.5\mathcal N(-0.4{f},I_{100})$, respectively. Here ${e} \in \mathbb{R}^{100}$ is a vector with all entries equal to $1$, and ${f} \in \mathbb{R}^{100}$ is a vector with the first $50$ entries equal to 1 and remaining $50$ entries equal to $-1$. Consider a setting with a small number of training samples $n_1 = n_2 = 10$, and then test on 1000 new samples from each mixture model. The radius of the uncertainty set and the kernel bandwidth are determined by cross-validation. 

%\begin{table}[H]\small\setlength\tabcolsep{2pt}
%\centering
%\vspace{-0.1in}
%\caption{\small GMM data, 100-dimensional, comparisons averaged over 500 trials.} \label{tab:risk_compare_gmm_highdim}
%%\renewcommand\arraystretch{0.9}
%\begin{tabular}{|m{5em}|c|c|c|c|c|c|c|c|c|c|}
%\specialrule{.08em}{0em}{0em}
%\# observation ($m$) &  1  &   2 &   3 &   4 &   5  &  6  &  7 &   8  &  9 &   10 \\
%\hline
%Test &  \bf{0.214} & \bf{0.215} & \bf{0.132}  & \bf{0.132}  & \bf{0.093}  & \bf{0.093}  & \bf{0.070}  & \bf{0.071}  & \bf{0.056} & \bf{0.056} \\
%\hline
%GMM & 0.259 & 0.260 & 0.176 & 0.176 & 0.131  & 0.132 & 0.103 & 0.104 & 0.09 &  0.085 \\
%\hline
%Logistic &  0.493 &  0.493& 0.491 & 0.491 & 0.489 & 0.488 & 0.488 &  0.488 & 0.487 & 0.487 \\ 
%\hline
%Kernel SVM & 0.356 & 0.358& 0.312 & 0.313 & 0.288 & 0.289 & 0.273 & 0.275 & 0.263 & 0.264\\
%\hline
%Three-layer perceptron & 0.416 & 0.416 & 0.380 & 0.381 & 0.358 & 0.357 & 0.340 & 0.340& 0.326 & 0.327 \\
%\specialrule{.08em}{0em}{0em}
%\end{tabular}
%\vspace{-0.1in}
%\end{table}

\begin{table}[ht]\small\setlength\tabcolsep{6pt}
\centering
\caption{ GMM data, 100-dimensional, comparisons averaged over 500 trials.} \label{tab:risk_compare_gmm_highdim}
\renewcommand\arraystretch{1.1}
\begin{tabular}{cccccc}
\specialrule{.08em}{0em}{0em}
\# observation ($m$) & Ours  & GMM & Logistic & Kernel SVM & 3-layer NN \\
\specialrule{.08em}{0em}{0em}
%1 & \bf{0.2141} & 0.2588 &  0.4925 & 0.3564 &  0.4164  \\
%2 & \bf{0.2147}  & 0.2597 &  0.4927 & 0.3581 & 0.4164  \\
%3 & \bf{0.1320}  &  0.1755 & 0.4905 &  0.3122  & 0.3796 \\
%4 & \bf{0.1321}  & 0.1762 & 0.4905 & 0.3129 & 0.3808 \\
%5 & \bf{0.0931} & 0.1310  &  0.4888 & 0.2877  & 0.3575 \\
%6 &\bf{0.0928}  & 0.1315  & 0.4881 & 0.2893 & 0.3570 \\
%7 & \bf{0.0702}  & 0.1034  & 0.4880 & 0.2727 & 0.3399 \\
%8 &\bf{0.0705}  & 0.1038  & 0.4876  & 0.2745  & 0.3401 \\
%9 & \bf{0.0563} & 0.0850   & 0.4873 & 0.2634  & 0.3264 \\
%10 & \bf{0.0564} &  0.0851 & 0.4874  & 0.2641 & 0.3267  \\
1 & \bf{0.2145} & 0.2588 &  0.4925 & 0.3564 &  0.4164  \\
2 & \bf{0.2157}  & 0.2597 &  0.4927 & 0.3581 & 0.4164  \\
3 & \bf{0.1331}  &  0.1755 & 0.4905 &  0.3122  & 0.3796 \\
4 & \bf{0.1329}  & 0.1762 & 0.4905 & 0.3129 & 0.3808 \\
5 & \bf{0.0937} & 0.1310  &  0.4888 & 0.2877  & 0.3575 \\
6 &\bf{0.0938}  & 0.1315  & 0.4881 & 0.2893 & 0.3570 \\
7 & \bf{0.0715}  & 0.1034  & 0.4880 & 0.2727 & 0.3399 \\
8 &\bf{0.0715}  & 0.1038  & 0.4876  & 0.2745  & 0.3401 \\
9 & \bf{0.0579} & 0.0850   & 0.4873 & 0.2634  & 0.3264 \\
10 & \bf{0.0578} &  0.0851 & 0.4874  & 0.2641 & 0.3267  \\
\specialrule{.08em}{0em}{0em}
\end{tabular}
\end{table}

We compare the performance of the proposed approach with several commonly used classifiers. They are comparable since binary classifiers can be used for deciding hypotheses, although they are designed with different targets. The competitors include the Gaussian Mixture Model (GMM), logistic regression, kernel support vector machine (SVM) with radial basis function (RBF) kernel, and a three-layer perceptron \cite{friedman2001elements} to illustrate the performance of neural networks. The results are summarized in Table~\ref{tab:risk_compare_gmm_highdim}, where the first column corresponds to the single observation scheme, while other columns are results using multiple observations, with the number of observations $m$ varying from $2$ to $10$. We use the majority rule for GMM, logistic regression, kernel SVM, and three-layer neural networks (NN) for testing batch samples. Note that there are over 2500 parameters in the neural network model with two hidden layers (50 nodes in each layer), which is challenging to learn when the training data size is small. Moreover, given only ten samples per class, estimating the underlying Gaussian mixture model is unrealistic, so that any parametric methods will suffer. The results demonstrate that when there is a small sample size, our minimax test outperforms other methods. 

% To illustrate the nature of the proposed test, we also plot the ``decision boundary'' by treating it as a binary classifier for  two-dimensional Gaussian distributions. As shown in Figure~\ref{fig:boundary}, when the kernel bandwidth $h$ is larger ($h=1$), the decision boundary is smoother, but it may misclassify some points that are close to the true boundary. When the kernel bandwidth $h$ becomes smaller ($h=0.1$) the decision boundary is sharper, but it can now classify the training data more correctly, and there seems to exist an optimal choice of the kernel bandwidth which we delegate to future work.
%  \begin{figure}[!ht]
%  \begin{center}
%  \includegraphics[width = 0.7\textwidth]{images/deci_2class.pdf} 
%  \caption{Decision boundary illustration for two-dimensional Gaussian distributions: two Gaussian distributions with mean values $[1,1]$ and $[-1,-1]$, respectively. The covariance matrix is $I_2$ for all distributions.} 
%  \label{fig:boundary}
%   \end{center}
%  \end{figure}

\subsection{Real data: MNIST handwritten digits classification}

We also compare the performance using MNIST handwritten digits dataset \cite{lecun1998gradient}. The full dataset contains 70,000 images, from which we randomly select five training images from each class.  We solve the optimal randomized test from \eqref{eq:wass vs wass:dual} with the radii parameters chosen by cross-validation. For the batch test setting, we divide test images from the same class into batches, each consisting of $m$ images. The decision for each batch is made using the majority rule for the optimal test in Section~\ref{sec:multi-obs}, as well as for logistic regression and SVM. We repeat this process to 500 randomly selected batches, and the average misclassification rates are reported in Table~\ref{tab:risk_compare_mnist_Kobs}.  The results show that our method significantly outperforms logistic regression and SVM. Moreover, the performance gain is higher in the batch test setting:  the errors decay quickly as $m$ increases. Note that the neural network-based deep learning model is not appropriate for this setting since the data-size is too small to train the model.

\begin{table}[ht]\small\setlength\tabcolsep{6pt}
\renewcommand\arraystretch{1.1}
\centering
\caption{ MNIST data, comparisons averages over 500 trials.}
\label{tab:risk_compare_mnist_Kobs}
\begin{tabular}{cccc}
\specialrule{.08em}{0em}{0em}
\# observation ($m$) & Ours  & Logistic & SVM  \\
\specialrule{.08em}{0em}{0em}
1 & \bf{0.3572} & 0.3729  &  0.3674   \\
2 & \bf{0.3631}   & 0.3797  &  0.3712    \\
3 & \bf{0.2772}   &  0.2897 & 0.2840  \\
4 &  \bf{0.2122}  & 0.2239 & 0.2169   \\
5 & \bf{0.1786} & 0.1882   &   0.1827 \\
6 &\bf{0.1540} & 0.1643  &  0.1588  \\
7 &  \bf{0.1347}  & 0.1446  & 0.1391 \\
8 &\bf{0.1185}  & 0.1276  & 0.1222  \\
9 & \bf{0.1063}  & 0.1160    & 0.1119  \\
10 & \bf{0.0960} &  0.1057  & 0.1010  \\
\specialrule{.08em}{0em}{0em}
\end{tabular}
\end{table}

\subsection{Application: Human activity detection}

In this subsection, we apply the optimal test for human activity detection from sequential data, using a dataset released by the Wireless Sensor Data Mining Lab in 2013 \cite{lockhart2011design, weiss2012impact, kwapisz2011activity}. In this dataset, 225 users were asked to perform specific activities, including walking, jogging, stairs, sitting, standing, and lying down; the data were recorded using accelerometers. Our goal is to detect the change of activity in real-time from sequential observations. Since it is difficult to build precise parametric models for distributions of various activities, traditional parametric change-point detection methods do not work well. We compare the proposed method with a standard nonparametric multivariate sequential change-point detection procedure based on the Hotelling's $T$-squared statistic \cite{montgomery2009introduction}. The raw data consists of sequences of observations for one person; each sequence may contain more than one change-points, and the time duration for each activity is also different. For this experiment, we only consider two types of transitions of activities: walking to jogging and jogging to walking. We extract 360 sequences of length 100 such that each sequence only contains one change-point.

We construct a change-point detection procedure using our optimal test as follows. Denote the data sequence as $\{\omega_t,t=1,2,\ldots\}$. At any possible change-point time $t$, we treat samples in time windows $[t-w,t-1]$ and $[t+1,t+w]$ as two groups of training data and find the LFDs $\{P_1^\ast,P_2^\ast\}$ by solving the convex problem in Equation \eqref{eq:wass vs wass:dual}. Then we calculate the detection statistic as $P_2^\ast(\omega_t) - P_1^\ast(\omega_t)$, inspired by the optimal detector in Lemma \ref{lemma:inf_phi}. %(the so-called $\ell_1$ detector \cite{gao2018robust}, which corresponds to an optimal  deterministic test that minimizing the hinge loss). 
We couple this test statistic with the CUSUM-type recursion \cite{page1954continuous}, which can accumulate change and detects small deviations quickly. The recursive detection statistic is defined as $S_t = \max\{0, S_{t-1} + P_2^\ast(\omega_t) - P_1^\ast(\omega_t)\}$, with $S_0=0$. A change is detected when $S_t$ exceeds a pre-specified threshold for the first time. Such scheme is similar to the combination of convex optimization solution and change-point detection procedure \cite{YangISIT2017}. 
In the experiment, we set the window size $w=10$ and choose the same radii for uncertainty sets using cross-validation. The Hotelling's $T$-squared procedure is constructed similarly. Using historical samples, we estimate the nominal (pre-change) mean $\widehat\mu$ and covariance $\widehat\Sigma$. The Hotelling's $T$-squared statistics at time $t$ is defined as $(\omega_t - \widehat\mu)^T\widehat\Sigma^{-1}(\omega_t - \widehat\mu)$ and the Hotelling procedure uses a CUSUM-type recursion: $H_t = \max\{0, H_{t-1}+(\omega_t - \widehat\mu)^T\widehat\Sigma^{-1}(\omega_t - \widehat\mu)\}$.

We compare the expected detection delay (EDD) versus Type-I error. Here EDD is defined as the average number of samples that a procedure needs before detects a change after it has occurred, which is a commonly used metric for sequential change-point detection \cite{xie2013sequential}. The Type-I error corresponds to the probability of detecting a change when there is no change. We consider a range of thresholds such that the corresponding Type-I error is from 0.05 to 0.35. The results in Figure~\ref{fig:detect_compare} show that our test significantly outperforms Hotelling's $T$-squared procedure in detecting the change quicker under the same Type-I error. 

\begin{figure}[!ht]
\begin{center}
\begin{tabular}{cc}
\includegraphics[width = 0.39\textwidth]{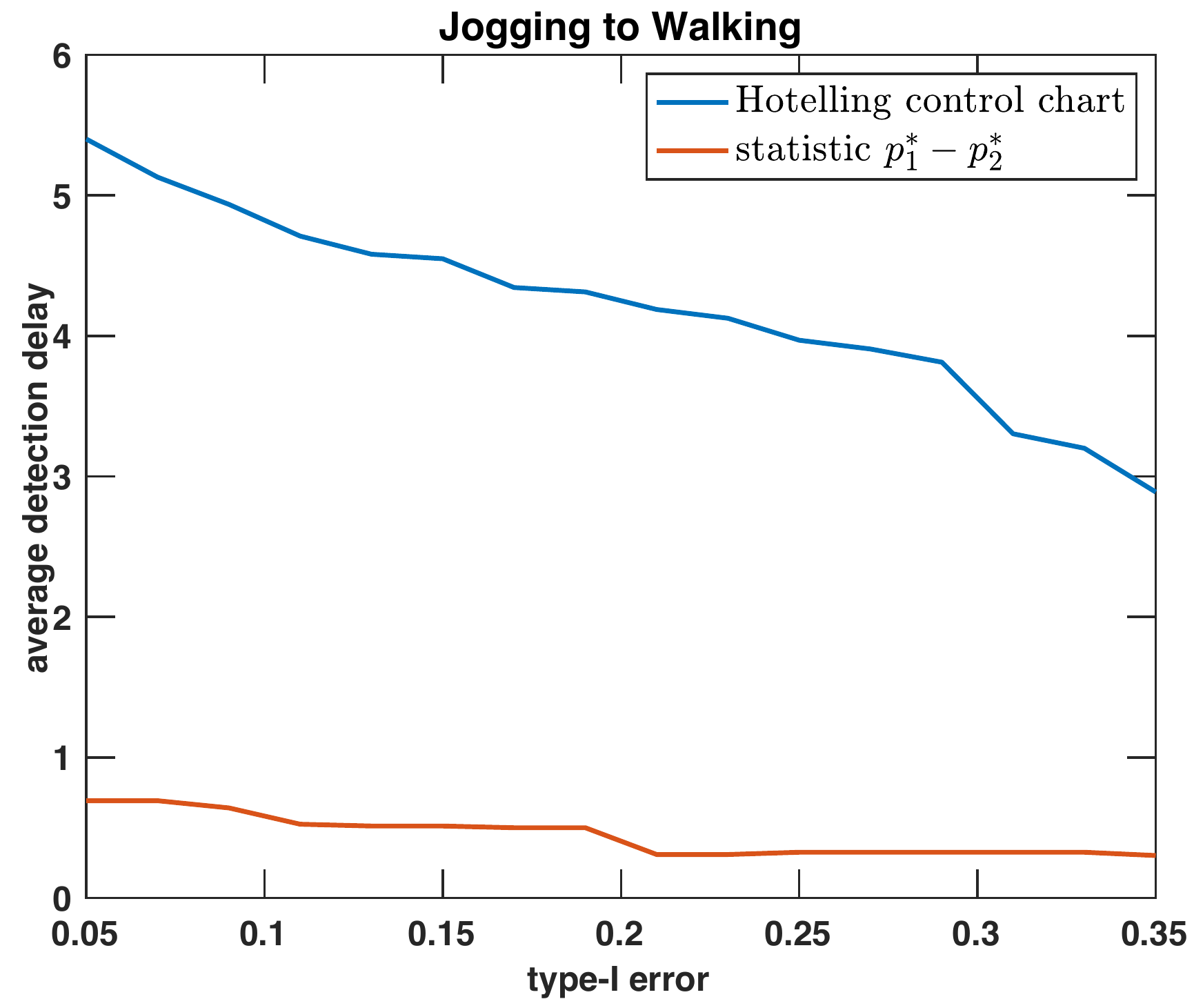} & \includegraphics[width = 0.4\textwidth]{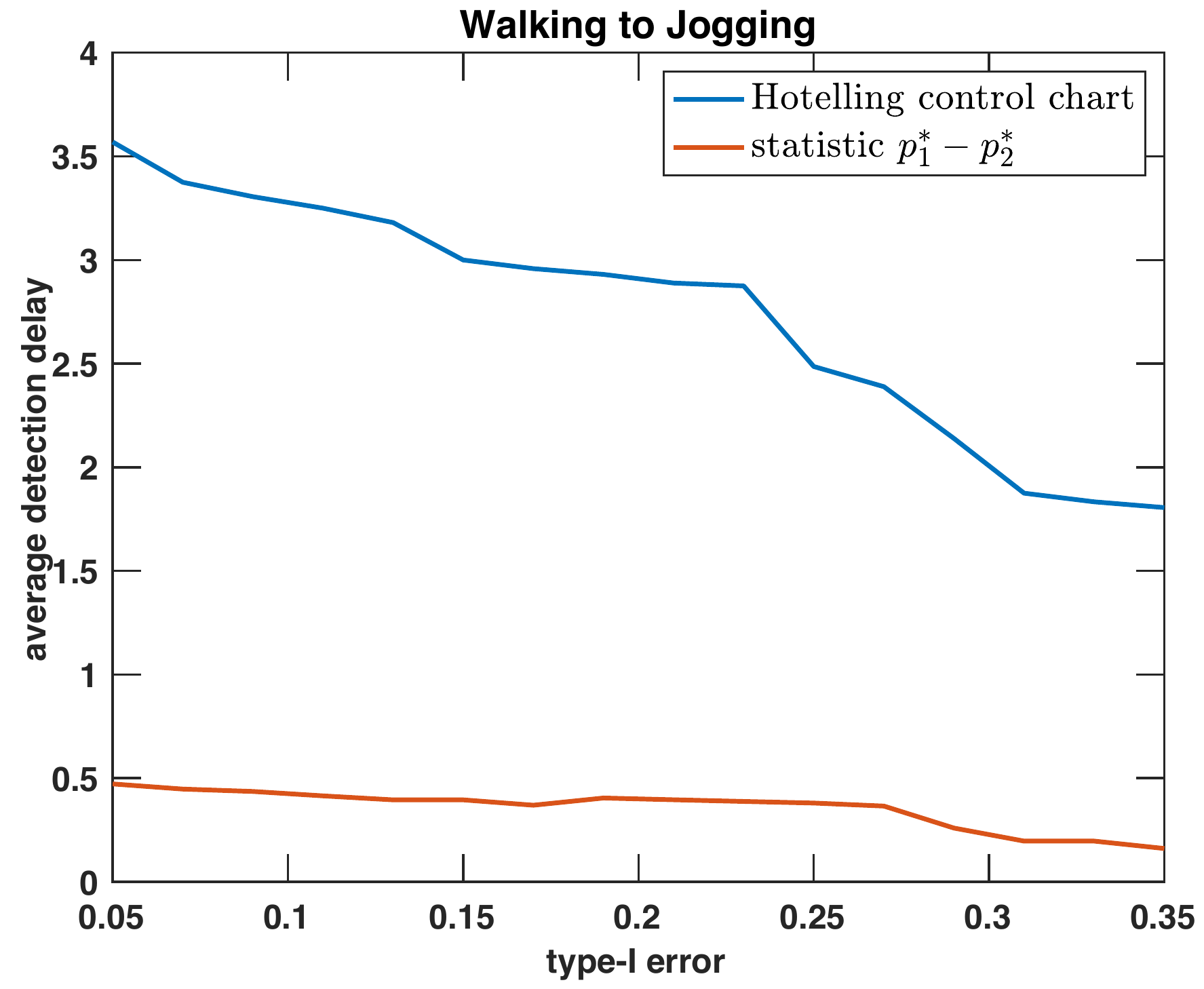}
\end{tabular}
\end{center}
\caption{Comparison of the Expected Detection Delay (EDD) of our test with the Hotelling's $T$-squared procedure for detecting two type of activity transitions: jogging to walking (left) and walking to jogging (right). 
%There are 72 sequences of walking to jogging and 43 sequences of jogging to walking, 
}
\label{fig:detect_compare}
\end{figure}

%\section{Extension to Neural-Nets-based Classification}
%
%Typically, a neural-net classifier solves the following problem
%\[
%  \min_{f,\{w_k\}_k}\quad \sum_{k=1}^2 \E_{x\sim Q_k}\left[-\log\left(\frac{\exp(w_k^\top f(x))}{\sum_{j}\exp(w_j^\top f(x))}\right)\right],
%\]
%where $f$ is the feature/representation learner that uses a neural net to discover the representations needed for classification from raw data $\{Q_k\}_k$; and $\{w_k\}_k$ are the weight vectors of a softmax function that classifies the sample $x$ using the representation $f(x)$ based on multinomial logistic regression.
%
%The robust classifier proposed in our paper solves
%\[
%  \min_{f,\{\pi_k\}_k} \max_{P_k \in \calP_k} \quad \sum_{k=1}^2 \E_{x\sim P_k} \left[1-\pi_k(f(x))\right],
%\]
%where we are searching for a classifier $\pi=(\pi_k)_k$ that minimizes the worst-case classification error.
%Alternatively, we can also formulate a problem that finds a classifier $\pi=(\pi_k)_k$ maximizing the worst-case likelihood function
%\[
%  \max_{f,\{\pi_k\}_k} \min_{P_k \in \calP_k} \quad \sum_{k=1}^2 \E_{x\sim P_k} \left[\log(\pi_k\circ f(x))\right],
%\]
%
%Applying the similar idea to GAN gives
%\[
%  \min_g \max_{f,\pi} \min_{P_g,Q}\ \ \E_{x\sim Q}[\log(\pi\circ f(x))] + \E_{z\sim \mathcal{N}(0,1)} [\log((1-\pi\circ f(g(z)))]
%\]

\section{Conclusions and Discussions}\label{sec:conc}

In this paper, we present a new approach for robust hypothesis testing when there are limited  ``training samples'' for each hypothesis. We formulate the problem as a minimax hypothesis testing problem to decide between two disjoint sets of distributions centered around empirical distributions in Wasserstein metrics. This formulation, although statistically sound -- can be treated as a ``data-driven'' version of Huber's robust hypothesis test, is computationally challenging since it involves an infinitely dimensional optimization problem. Thus, we present a computationally efficient framework for solving the minimax test, revealing the optimal test's statistical meaning. We also prove how to extend the minimax test from empirical support to the whole space and use it for the ``batch'' test settings. Moreover, we characterize the radius selection by providing an asymptotic upper bound for the sufficient radii and shed light on the optimal test's generalization property. We demonstrate the good performance of the proposed robust test on simulated and real data. 

The method can be kernelized to handle more complex data structures (e.g., the observations are not real-valued). The kernelization can be conveniently done by replacing the metric $c(\cdot,\cdot)$ used in solving the optimal test \eqref{eq:wass vs wass:dual} with other distances metrics between features after kernel transformation. Take the Euclidean norm as an example. Given a kernel function $\mathcal K(\cdot,\cdot)$ that measures similarity between any pair of data, the pairwise norm $c(\omega^l,\omega^m)=\norm{\omega^l-\omega^m}$ in \eqref{eq:wass vs wass:dual} can be replaced with the kernel version distance $\mathcal K(\omega^l, \omega^m)$. Moreover, this means that the proposed framework can be combined with feature selection and neural networks to enhance its performance in practice for complex datasets.

\section*{Acknowledgements}
%See \ref{suppA} for the supplementary material example.

The work of Liyan Xie and Yao Xie are funded by NSF CAREER CCF-1650913, DMS-1938106, DMS-1830210, and CMMI-2015787.

\bibliographystyle{plain}
\bibliography{RobustChangeDetection_Was,CovDetection}

%\begin{supplement}
%\sname{Supplement A}\label{suppA}
%\stitle{Proofs}
%\slink[url]{http://www.e-publications.org/ims/support/dowload/imsart-ims.zip}

\clearpage
\appendix

\section{Proofs for Section~\ref{sec:convex detector}}

\subsection{Proof of Lemma~\ref{lemma:inf_phi}}\label{proof:inf_phi}
  Note that the probability measures $P_1$, $P_2$ are absolutely continuous with respect to $P_1+P_2$, hence we have
    \begin{equation}\label{eq:int_inf}
      \begin{aligned}
     & \inf_{\pi} \Phi(\pi;P_1,P_2) \\
      =& \inf_{\pi} \int_{\Omega} \big[(1-\pi(\omega))\textstyle {\frac{dP_1}{d(P_1+P_2)}(\omega)} + \pi(\omega)\textstyle {\frac{dP_2}{d(P_1+P_2)}(\omega)} \big]d(P_1+P_2)(\omega) \\
       =& \inf_{\pi} \int_{\Omega_0} \big[(1-\pi(\omega))\textstyle {\frac{dP_1}{d(P_1+P_2)}(\omega)} + \pi(\omega)\textstyle {\frac{dP_2}{d(P_1+P_2)}(\omega)} \big]d(P_1+P_2)(\omega) \\
      = & \int_{\Omega_0} \inf_{0\le x\le 1} \big[(1-x)\textstyle {\frac{dP_1}{d(P_1+P_2)}(\omega)} + x\textstyle {\frac{dP_2}{d(P_1+P_2)}(\omega)} \big] d(P_1+P_2)(\omega),
      \end{aligned}
    \end{equation}
  where the second equality holds because the integral depends only on the subset $\Omega_0 := \big\{\omega\in\Omega: \textstyle 0<\frac{dP_k}{d(P_1+P_2)}(\omega) < 1,k=1,2 \big\}$, on which $P_1$, $P_2$ are absolutely continuous with respect to each other; the third equality is due to Lemma \ref{lem:change}, with $\cal M$ being the set of measurable functions and $f(x,\omega)= \big[(1-x)\textstyle {\frac{dP_1}{d(P_1+P_2)}(\omega)} + x\textstyle {\frac{dP_2}{d(P_1+P_2)}(\omega)} \big]\chi_{[0,1]}(x) $, where $\chi_{[0,1]}(x)=1$ if $x\in[0,1]$ and $\infty$ otherwise. 

  For any $\omega$, the infimum $\pi^*(\omega)$ of the inner minimization in \eqref{eq:int_inf} is attained at $0$ or $1$. Therefore, for any $\omega\in\Omega$,
  \[
    (1-\pi^*(\omega))\textstyle {\frac{dP_1}{d(P_1+P_2)}(\omega)} + \pi^*(\omega)\textstyle {\frac{dP_2}{d(P_1+P_2)}(\omega)} =
    \min\left\{\frac{dP_1}{d(P_1+P_2)}(\omega),\frac{dP_2}{d(P_1+P_2)}(\omega)\right\}.
  \]
    This completes the proof.

\subsection{Proof of Lemma~\ref{lemma:wass vs wass}}\label{proof:wass vs wass}
  Denote by $L^1(\mu)$ the space of all integrable functions with respect to the measure $\mu$. Using Lagrangian and Kantorovich's duality (Lemma~\ref{lem:kan}), we rewrite the problem as
    \allowdisplaybreaks
    \[\label{eq:kan_duality}
      \begin{multlined}
        \sup_{P_1\in\calP_1,P_2\in\calP_2} \psi(P_1,P_2)\\
       \shoveleft{=  \sup_{\substack{P_1\in \scrP(\Omega) \\ P_2\in\scrP(\Omega)}}\inf_{\lambda_1,\lambda_2 \geq 0 } \Bigg\{ \psi(P_1,P_2) + \sum_{k=1}^2\lambda_k\theta_k- \sum_{k=1}^2\!\lambda_k \sup_{\substack{u_k\in \R^{n_k}\\v_k\in L^1(P_k)}}  \bigg\{\frac{1}{n_k}\sum_{i=1}^{n_k} u_k^i  }\\
        \shoveright{+ \int_\Omega v_k dP_k: u_k^i+v_k(\omega)\leq c(\omega,\homega_k^i),\ \forall 1\leq i\leq n_k, \forall \omega\in\Omega \bigg\}\Bigg\}} \\
     \shoveleft{= \sup_{\substack{P_1\in \scrP(\Omega) \\ P_2\in\scrP(\Omega)}} \inf_{\substack{\lambda_1,\lambda_2 \geq 0\\u_k\in \R^{n_k}\\v_k\in L^1(P_k)}} \Bigg\{\psi(P_1,P_2) + \sum_{k=1}^2\lambda_k\theta_k - \sum_{k=1}^2 \lambda_k\bigg( \frac{1}{n_k}\sum_{i=1}^{n_k} u_k^i + \int_\Omega v_k dP_k \bigg) :} \\
    \shoveright{ u_k^i+v_k(\omega)\leq c(\omega,\homega_k^i),\ \forall 1\leq i\leq n_k, \forall \omega \in \Omega \Bigg\} } \\
    \shoveleft{ = \sup_{\substack{P_1\in \scrP(\Omega) \\ P_2\in\scrP(\Omega)}}  \inf_{\substack{\lambda_1,\lambda_2 \geq 0\\u_k\in \R^{n_k}\\v_k\in L^1(P_k)}} \Bigg\{\psi(P_1,P_2) +  \sum_{k=1}^2\lambda_k\theta_k - \sum_{k=1}^2 \bigg( \frac{1}{n_k} \sum_{i=1}^{n_k} u_k^i + \int_\Omega v_k dP_k \bigg): } \\
    u_k^i+v_k(\omega)\leq \lambda_k c(\omega,\homega_k^i), \ \forall 1\leq i\leq n_k, \forall \omega \in \Omega \Bigg\},
    \end{multlined}
  \]
  where the second equality holds by combining the innermost supreme problem with the infimum problem; and the third equality holds by replacing $\lambda_k u_k^i$ with $u_k^i$ and $\lambda_k v_k$ with $v_k$ (note that such change of variable is valid even when $\lambda_k=0$). 
  Furthermore, since the objective function is non-increasing in $v_k$, we can replace $v_k$ with $\min_{1\leq i\leq n_k} \{\lambda_k c(\omega,\homega_k^{i}) - u_k^{i}\}$ without changing the optimal value.
  Interchanging $\sup$ and $\inf$ yields
  \begin{equation}\label{eq:wass vs wass_interchange}
    \begin{multlined}
    \sup_{P_1\in\calP_1,P_2\in\calP_2}  \psi(P_1,P_2)\\
     \shoveleft{\leq\inf_{\substack{\lambda_1,\lambda_2 \geq 0\\u_k\in \R^{n_k}}}  \bigg\{\sum_{k=1}^2 \lambda_k\theta_k -  \sum_{k=1}^2 \frac{1}{n_k} \sum_{i=1}^{n_k} u_k^i + \sup_{\substack{P_1\in \scrP(\Omega) \\ P_2\in\scrP(\Omega)}}  \Big\{  \psi(P_1,P_2)} \\   
     - \int_\Omega \sum_{k=1}^2 \min_{1\leq i\leq n_k} \{\lambda_k c(\omega,\homega_k^{i}) - u_k^{i}\} dP_k \Big\} \bigg\}.        
     \end{multlined}
  \end{equation}
 
  Now let us study the inner supremum in \eqref{eq:wass vs wass_interchange}. For a given distribution $(P_1,P_2)$ and any $\omega\in \supp P_1\cup\supp P_2$, let $i_k(\omega) = \argmin_{i}\{\lambda_k c(\omega,\homega_k^{i}) - u_k^{i}\}$, $k=1,2$, set
  \[
  \begin{aligned}
    T(\omega) 
    :=& \begin{cases} \homega_1^{i_1(\omega)}, & \text{if } \lambda_1\textstyle \frac{dP_1}{d(P_1+P_2)}(\omega) \geq \lambda_2\textstyle \frac{dP_2}{d(P_1+P_2)}(\omega),  \\
    \homega_2^{i_2(\omega)}, & \text{if } \lambda_1\textstyle \frac{dP_1}{d(P_1+P_2)}(\omega) < \lambda_2\textstyle \frac{dP_2}{d(P_1+P_2)}(\omega),
    \end{cases}
    \end{aligned}
  \]
  whence
  \[
   T(\omega) \in \argmin_{\omega'\in\Omega} \bigg\{ 
    \sum_{k=1}^2 \big[\lambda_k c(\omega',\homega_k^{i_k(\omega)}) - u_k^{i_k(\omega)}\big] \textstyle \frac{dP_k}{d(P_1+P_2)}(\omega) \bigg\}.
  \]
  By definition we have $T(\omega)\in\hOmega$. 
  Define another solution $(P'_1,P'_2)$ such that $P'_k(B)=P_k\{\omega\in\Omega:T(\omega)\in B\}$ for any Borel set $B\subset\hOmega$.
  It follows that
  \[\begin{aligned}
     &\sum_{k=1}^2  \int_{\hOmega} \min_{1\leq i\leq n_k} \{\lambda_k c(\omega,\homega_k^{i}) - u_k^{i}\} dP'_k(\omega) \\
     = & \sum_{k=1}^2  \int_{\Omega}  \min_{1\leq i\leq n_k} \{\lambda_k c(T(\omega),\homega_k^{i}) - u_k^{i}\} dP_k(\omega)\\
     \leq & \sum_{k=1}^2 \int_\Omega \big(\lambda_k c(T(\omega),\homega_k^{i_k(\omega)}) - u_k^{i_k(\omega)}\big) dP_k(\omega) \\
     \leq &\sum_{k=1}^2   \int_\Omega \big(\lambda_k c(\omega,\homega_k^{i_k(\omega)}) - u_k^{i_k(\omega)}\big) dP_k(\omega).
  \end{aligned}
  \]
  In addition, by a simple fact that $ \sum_i \min\{x_i,y_i\}\leq \min \{\sum_ix_i, \sum_i y_i\}$ for any series $\{x_i,y_i\}$, 
  we have
  \[
    \begin{aligned}
     \psi(P_1,P_2)  & = \int_{\Omega}\min\left\{\frac{dP_1}{d(P_1+P_2)}(\omega),\frac{dP_2}{d(P_1+P_2)}(\omega)\right\}\, d(P_1+P_2)(\omega) \\
     & \leq \sum_{\homega\in\hOmega} \min\{P_1\{\omega\in\Omega: T(\omega)=\homega\} ,P_2\{\omega\in\Omega: T(\omega)=\homega\}  \}\\
     & = \sum_{\homega\in\hOmega} \min\{P_1'(\homega) ,P_2'(\homega) \}\\
     & =  \psi(P_1',P_2').
  \end{aligned}
  \]
  Hence $(P'_1,P'_2)$ yields an objective value no worse than $(P_1,P_2)$ for the inner supremum in \eqref{eq:wass vs wass_interchange}.
  This suggests that in order to solve the inner supremum of \eqref{eq:wass vs wass_interchange}, it suffices to only consider $(P_1,P_2)$ with $\supp P_1 \subset \hOmega$ and $\supp P_2 \subset \hOmega$.

  For $l=1,\ldots,n$, set $p_k^l = P_k(\homega^l)$, and note that $\gamma_k\in\Gamma(P_k,Q_{k,n}) $ can be identified with a non-negative matrix $\gamma_k\in\R_+^{n\times n}$ with each column and row summing up to 1.
  Thus, the inner supremum in \eqref{eq:wass vs wass_interchange} can now be equivalently written as
  \[
   \sup_{\substack{p_1,p_2\in\R_+^{n}\\\sum_l p_1^l=1,\,  \sum_l p_2^l=1}} \bigg\{ \sum_{l=1}^n \min\big\{ p_1^l, p_2^l\big\} - \sum\limits_{k=1}^2 \sum\limits_{l=1}^{n} p_k^l \min_{1\leq i\leq n_k} \{\lambda_k c(\homega^l,\homega_k^i) - u_k^i\} \bigg\}.
  \]
  It follows that
  \[
    \begin{multlined}
      \sup_{P_1\in\calP_1,P_2\in\calP_2} \psi(P_1,P_2)\\
      \shoveleft{ \leq \inf_{\lambda_1,\lambda_2\geq0} \bigg\{ \sum_{k=1}^2\lambda_k\theta_k -  \sum_{k=1}^2 \frac{1}{n_k} \sum_{i=1}^{n_k} u_k^i + \sup_{\substack{p_1,p_2\in\R_+^{n}\\\sum_l p_1^l=1,\,  \sum_l p_2^l=1}} \bigg\{ \sum_{l=1}^n \min\big\{ p_1^l, p_2^l\big\} } \\ 
    - \sum_{k=1}^2 \sum_{l=1}^{n} p_k^l \min_{1\leq i\leq n_k} \{\lambda_k c(\homega^l,\homega_k^i) - u_k^i\}\bigg\}.
    \end{multlined}
  \]
  Applying the Lagrangian duality for finite-dimensional convex programming on the right-hand side yields
  \[
     \sup_{P_1\in\calP_1,P_2\in\calP_2} \psi(P_1,P_2) \leq  \sup_{\substack{P_1\in\bcalP_1,P_2\in\bcalP_2}} \psi(P_1,P_2),
  \]
  where $\bcalP_k := \calP_k \cap \scrP(\hOmega)$, $k= 1, 2$.
  Observe that both sides have the same objective function, but the feasible region of the right-hand side is a subset of that of the left-hand side, and thus the right-hand side should be no greater than the left-hand side, i.e., the above inequality should hold as equality.
  Thereby we complete the proof.

\subsection{Proof of Theorem~\ref{thm:opt_test}}\label{proof:opt_test}   

  Note that from Lemma \ref{lemma:wass vs wass}, we have
  \[\begin{aligned}
    \sup_{P_1\in \bcalP_1, P_2\in \bcalP_2} \inf_{\pi:\Omega\to[0,1]}  \Phi(\pi;P_1,P_2) 
    &= \sup_{P_1\in \calP_1, P_2\in \calP_2} \inf_{\pi:\Omega\to[0,1]}  \Phi(\pi;P_1,P_2) \\
    & \leq \inf_{\pi:\Omega\to[0,1]}  \sup_{P_1\in \calP_1, P_2\in \calP_2} \Phi(\pi;P_1,P_2).
    \end{aligned}
  \]
  Let us prove the other direction.

  To begin with, we identify $\hpi\in[0,1]^n$ with a function on $\hOmega$.
  Using Lemma \ref{lemma:DRO}, we have
  \begin{equation}\label{eq:DRO:hat}
    \begin{aligned}
    \sup_{P_1\in\bcalP_1} \mathbb E_{P_1} [1-\hpi] 
    & =  \inf_{\lambda_1\geq0} \bigg\{\lambda_1 \theta_1 +  \frac{1}{n_1} \sum_{l=1}^{n_1} \max_{1\leq m\leq n} \{1-\hpi_m - \lambda_1 c(\homega^l,\homega^m)\} \bigg\},\\
    \sup_{P_2\in\bcalP_2}  \mathbb E_{P_2}[\hpi] & = \inf_{\lambda_2\geq0} \bigg\{\lambda_2 \theta_2 +  \frac{1}{n_2} \sum_{l=1}^{n_2} \max_{1\leq m\leq n} \{\hpi_m - \lambda_2 c(\homega^l,\homega^m)\} \bigg\}. 
    \end{aligned}
  \end{equation}
  Let $\lambda_1^*$ and $\lambda_2^*$ be respectively the minimizers of the two problems in \eqref{eq:DRO:hat}.
  Observe that the right-hand sides of \eqref{eq:DRO:hat} and \eqref{eq:opt_test_constr} are identical.
  Hence \eqref{eq:opt_test_constr} implies that $\hpi^\ast$ defined in the statement of Theorem \ref{thm:opt_test} satisfies
  \[
    \E_{P_1^\ast}[1-\hpi^\ast] = \sup_{P_1\in\bcalP_1} \mathbb E_{P_1} [1-\hpi^\ast] ,\ \ 
    \E_{P_2^\ast}[\hpi^\ast] = \sup_{P_2\in\bcalP_2}  \mathbb E_{P_2}[\hpi^\ast],
  \]
  and thus 
  \begin{equation}\label{eq:saddle_point_hatOmega}
    \sup_{P_1\in \bcalP_1, P_2\in \bcalP_2}  \Phi(\hpi^\ast;P_1,P_2) = \sup_{P_1\in \bcalP_1, P_2\in \bcalP_2} \inf_{\hpi\in[0,1]^n}  \Phi(\hpi;P_1,P_2). 
  \end{equation}
  Hence $(\hpi^\ast;P_1^\ast,P_2^\ast)$ solves the above finite-dimensional convex-concave saddle point problem that always has an optimal solution, which verifies the well-definedness of $\hpi^\ast$.
  
  On the other hand, for the $\pi^\ast$ defined in the statement of Theorem \ref{thm:opt_test}, the optimization problem for finding worst-case risk are decoupled and admits the following equivalent reformulations (Lemma \ref{lemma:DRO})
  \begin{equation}\label{eq:subproblem_1}
    \begin{aligned}
    \sup_{P_1\in\calP_1}  \mathbb E_{P_1}[1-\pi^\ast(\omega)] & = \min_{\lambda_1\geq0}\left\{ \lambda_1 \theta_1 + \frac{1}{n_1}\sum_{i=1}^{n_1}\sup_{\omega\in\Omega}\left\{1-\pi^\ast(\omega) - \lambda_1 c(\omega,\homega_1^i)  \right\}  \right\},\\
    \sup_{P_2\in\calP_2}  \mathbb E_{P_2}[\pi^\ast(\omega)] & = \min_{\lambda_2\geq0}\left\{ \lambda_2 \theta_2 + \frac{1}{n_2}\sum_{i=1}^{n_2} \sup_{\omega\in\Omega}\left\{ \pi^\ast(\omega)-\lambda_2 c(\omega,\homega_2^i) \right\}  \right\}. 
    \end{aligned}
  \end{equation}    
  Comparing \eqref{eq:DRO:hat} and \eqref{eq:subproblem_1}, if we can prove $\pi^\ast$ satisfies 
  \begin{equation}\label{eq:strong dual:condition}
    \begin{aligned}
    &\sup_{\omega\in\Omega}\left\{1-\pi^\ast(\omega) - \lambda_1^* c(\omega,\homega_1^i)  \right\} \leq \max_{\omega\in\hOmega} \left\{1-\hpi^\ast(\omega) - \lambda_1^* c(\omega,\homega_1^i)\right\}, \ \forall 1\leq i\leq n_1,\\
    &\sup_{\omega\in\Omega}\left\{ \pi^\ast(\omega)-\lambda_2^* c(\omega,\homega_2^i) \right\}  \leq \max_{\omega\in\hOmega}\left\{\hpi^\ast(\omega) - \lambda_2^* c(\omega,\homega_2^i) \right\},\ \forall 1\leq i\leq n_2, 
    \end{aligned}
  \end{equation}
  then $\pi^\ast$ would be an optimal solution to \eqref{problem:HT_detector} since
  \[\begin{aligned}
    \inf_{\pi:\Omega\to[0,1]}  \sup_{P_1\in \calP_1, P_2\in \calP_2} \Phi(\pi;P_1,P_2) & \leq \sup_{P_1\in \calP_1, P_2\in \calP_2} \Phi(\pi^\ast;P_1,P_2)\\
    &\leq \sup_{P_1\in \bcalP_1, P_2\in \bcalP_2}  \Phi(\hpi^\ast;P_1,P_2) \\
    & = \sup_{P_1\in \calP_1, P_2\in \calP_2} \inf_{\pi:\Omega\to[0,1]}  \Phi(\pi^\ast;P_1,P_2).
    \end{aligned}
  \]

  To show \eqref{eq:strong dual:condition}, for $\pi^*$ restricted on the empirical support $\hOmega$, we have
  \[
    \begin{aligned}
    &\sup_{\omega\in\hOmega}\left\{1-\pi^\ast(\omega) - \lambda_1^* c(\omega,\homega_1^i)  \right\} = \max_{\omega\in\hOmega} \left\{1-\hpi^\ast(\omega) - \lambda_1^* c(\omega,\homega_1^i)\right\}, \ \forall 1\leq i\leq n_1,\\
    &\sup_{\omega\in\hOmega}\left\{ \pi^\ast(\omega)-\lambda_2^* c(\omega,\homega_2^i) \right\} = \max_{\omega\in\hOmega}\left\{\hpi^\ast(\omega) - \lambda_2^* c(\omega,\homega_2^i) \right\},\ \forall 1\leq i\leq n_2.
    \end{aligned}
  \]
  Indeed, this holds by construction $\pi^\ast(\omega)=\hpi^\ast(\omega)$ for $\omega\in\hOmega$.
  It remains to show \eqref{eq:strong dual:condition} also holds outside of $\hOmega$:
  \[\label{eq:strong dual:condition:not in hOmega}
    \begin{aligned}
    &\sup_{\omega\notin\hOmega}\left\{1-\pi^\ast(\omega) - \lambda_1^\ast c(\omega,\homega_1^i)  \right\} \leq \max_{\omega\in\hOmega} \left\{1-\hpi^\ast(\omega) - \lambda_1^\ast c(\omega,\homega_1^i)\right\}, \ \forall 1\leq i\leq n_1,\\
    &\sup_{\omega\notin\hOmega}\left\{ \pi^\ast(\omega)-\lambda_2^\ast c(\omega,\homega_2^i) \right\}  \leq \max_{\omega\in\hOmega}\left\{\hpi^\ast(\omega) - \lambda_2^\ast c(\omega,\homega_2^i) \right\},\ \forall 1\leq i\leq n_2. 
    \end{aligned}
  \]
  %   To show \eqref{eq:strong dual:condition:hOmega}, by optimality of  $\gamma_1^*,\gamma_2^*$ in \eqref{eq:wass vs wass:dual}, $\pi^*(\homega^m)=\hpi^\ast_m$ satisfies that for any $\homega^l$ with $\gamma_{1,l,m}^* \neq 0$, 
    %   \begin{equation}\label{tranport}
    %     1-\pi^*(\homega^m) - \lambda_1^* c(\homega^l,\homega^m) = \max_{\omega \in \hOmega} \{ 1-\pi^*(\omega) - \lambda_1^* c(\omega,\homega^l)\}.
    %     \end{equation} 
    %     And similarly, $\pi^*(\homega^m)$ satisfies that for any $\homega^l$ with $\gamma_{2,l,m}^* \neq 0$,
    %     \begin{equation}\label{tranport2}
    %     \pi^*(\homega^m) - \lambda_2^* c(\homega^l,\homega^m) = \max_{\omega \in \hOmega} \{ \pi^*(\omega) - \lambda_2^* c(\omega,\homega^l)\}.
    %   \end{equation}
    %   This is essentially a system of linear inequalities, from which we can solve the optimal test value $\{\pi^*(\homega^m), m=1,\ldots,n\}$ efficiently. And the optimal value always exists due to the existence of saddle point for finite-dimensional convex-concave minimax problem. 
  To prove this, note that it is equivalent to that $\forall \omega\notin\hOmega$: 
  \begin{equation}\label{eq:k1constraint}
    \begin{aligned}
    & \pi^*(\omega) \geq \min_{\homega\in\hOmega}\left\{\pi^*(\homega) + \lambda_1^* c(\homega,\homega_1^i)  \right\} - \lambda_1^* c(\omega,\homega_1^i) , \quad \forall i = 1,\ldots, n_1, \\
    & \pi^*(\omega) \leq \lambda_2^* c(\omega,\homega_2^j) -  \min_{\homega\in\hOmega}\left\{\lambda_2^* c(\homega,\homega_2^j) - \pi^*(\homega) \right\},\quad \forall j = 1,\ldots, n_2.
    \end{aligned}
  \end{equation}
  Observe that $\forall i=1,\ldots,n_1$ and  $\forall j=1,\ldots,n_2$, we have: 
  \[
    \begin{aligned}
      &\hphantom{\leq} \min_{\homega\in\hOmega}\left\{\pi^*(\homega) + \lambda_1^* c(\homega,\homega_1^i)  \right\} 
      + \min_{\homega\in\hOmega}\left\{\lambda_2^* c(\homega,\homega_2^j) - \pi^*(\homega) \right\}\\
      & \leq \begin{cases}
        \pi^\ast(\homega_2^j) + \lambda_1^\ast c(\homega_2^j,\homega_1^i) - \pi^\ast(\homega_2^j), & \lambda_1^\ast\leq \lambda_2^\ast,\\
        \pi^\ast(\homega_1^i) + \lambda_2^\ast c(\homega_1^i,\homega_2^j) - \pi^\ast(\homega_1^i), & \lambda_1^\ast > \lambda_2^\ast,
      \end{cases}\\
    & = \min\{\lambda_1^\ast,\lambda_2^\ast\} c(\homega_1^i,\homega_2^j) \\
    & \leq \lambda_1^\ast c(\omega,\homega_1^i) + \lambda_2^\ast c(\omega,\homega_2^j),\quad\forall\omega\in\Omega,
    \end{aligned}
  \]
  where we have used the triangle inequality of $c$.
  And we note that
  \[
  \min_{\homega\in\hOmega}\left\{\pi^*(\homega) + \lambda_1^* c(\homega,\homega_1^i)  \right\} - \lambda_1^* c(\omega,\homega_1^i)  \leq \pi^*(\homega_1^i) \leq 1,
  \]
  and 
  \[
  \lambda_2^* c(\omega,\homega_2^j) -  \min_{\homega\in\hOmega}\left\{\lambda_2^* c(\homega,\homega_2^j) - \pi^*(\homega) \right\} \geq \pi^*(\homega_2^j) \geq 0,
  \]
  since $\pi^\ast(\omega)=\hpi^\ast(\omega) \in [0,1]$ for $\omega\in\hOmega$. 
  Therefore we always have $l(\omega)\leq u(\omega)$ and \eqref{eq:k1constraint} always admits a feasible solution, as defined in the Theorem statement.

\subsection{Proof of Proposition~\ref{multi_obs_risk}} 
  
Given batch samples $\omega_1,\ldots,\omega_m$ sampled i.i.d. from the true distribution $P_1^\circ$, define Boolean random variables $\xi_i, 1\leq i \leq m$ as:
  \[
  \xi_i = \begin{cases} 1 & \pi_1(\omega_i) = 0; \\
  0 & \pi_1(\omega_i) = 1, \\
  \end{cases}
  \]
  more specifically, the random variable $\xi_i=1$ if and only if the test, as applied to observation $\omega_i$, rejects hypothesis $H_0$. 
  
  Further, by construction of the Majority test, if the hypothesis $H_0$ is rejected, then the number of $i$'s with $\xi_i$ = 1 is at least $m/2$. Thus, the probability to reject $H_0$ is not greater than the probability of the event: in $m$ random Bernoulli trials with probability $\epsilon^\ast$ of success, the total number of successes is $\geq m/2$. The probability of this event clearly does not exceed:
\[
  \sum_{m/2 \leq i \leq m} {m \choose i} (\epsilon^\ast)^i(1-\epsilon^\ast)^{m-i}.
\]
When $\epsilon^\ast < 1/2$, by the Chernoff bound, we have
\[
\sum_{m/2 \leq i \leq m} {m \choose i} (\epsilon^\ast)^i(1-\epsilon^\ast)^{m-i}\leq \exp\left\{- D(1/2 || \epsilon^*)m\right\},
\]
where $D(1/2 || \epsilon^*):=\frac12\log\frac{1}{2\epsilon^*}+\frac12\log\frac{1}{2(1-\epsilon^*)}$ is the relative entropy between two Bernoulli distributions with ``success'' probabilities being $1/2$ and $\epsilon^*$ respectively. It is easy to see that $D(1/2 || \epsilon^*)>0$. Therefore, the risk goes to 0 exponentially fast, in the order of $\exp\{- D(1/2 || \epsilon^*)m\} $ as $m \rightarrow \infty$.

\section{Proofs for Section~\ref{sec:radius}}\label{app:proof_sec_4}

%We assume that those subsets only overlap on their boundaries. For example, when true distributions are Gaussian, the overlapping area is an ellipse with measure-zero under true distributions. 
%\begin{assumption}[Boundary]\label{ass:boundary}
%The true distributions $P_1^\circ,\ldots,P_K^0$ satisfy the following: 
%\[
%\mathbb P_k^\circ [ P_1^\circ(\omega) \neq P_2^\circ(\omega) \neq\ldots \neq P_K^0(\omega) ] =1, \,\forall k
%\] 
%\end{assumption}
%\begin{assumption}[Moment]\label{ass:moments}
%The true distributions $P_1^\circ,\ldots,P_K^0$ satisfy moment constraint of 2: 
%\[
%\E_{\xi \sim P_k^\circ} [ c(\xi,\zeta)^2 ] < +\infty, \forall \zeta\in\Omega, \forall k.
%\] 
%\end{assumption}

\subsection{Proof of Lemma~\ref{lemma:profile_dual}}\label{sec:B.1}
\leavevmode

We first establish an optimality condition (Lemma~\ref{lem:lagrangian}) for the constraint 
$$
\pi^\circ \in \argmin_{\pi:\Omega\to[0,1]} \Phi(\pi;P_1,P_2).
$$
Without causing confusion, we simply write $
\pi^\circ \in \argmin_{\pi} \Phi(\pi;P_1,P_2)$ in subsequent proofs.
%We assume $\Omega_1^\circ$ and $\Omega_2^\circ$ are measurable sets and thus the oracle test $\pi^\circ$ is a measurable function. 

\begin{lemma}\label{lem:lagrangian}
Let $\pi^\circ$ be the oracle test. For any $P_1,P_2\in\scrP(\Omega)$, the constraint 
$$
\pi^\circ \in \argmin_{\pi:\Omega\to[0,1]} \Phi(\pi;P_1,P_2)
$$ 
holds if and only if 
\begin{equation}\label{eq:radius_constr}
\sup_{\bm\alpha_1,\bm\alpha_2 \in \mathcal B_+(\Omega)} \int_\Omega [\bm\alpha_2(\omega) \ind_{\Omega_2^\circ}(\omega) - \bm\alpha_1(\omega) \ind_{\Omega_1^\circ}(\omega)] (dP_1 - dP_2)(\omega) = 0.
\end{equation}
\end{lemma}
\begin{proof}
We first prove the necessity. Suppose $\pi^\circ \in \argmin_{\pi} \Phi(\pi;P_1,P_2)$. Then by definition for all randomized test $\pi$, we have
\[
  \Phi(\pi;P_1,P_2) \geq \Phi(\pi^\circ;P_1,P_2).
\] 
For any $\bm\alpha_1,\bm\alpha_2\in \mathcal B_+(\Omega)$, there exists a small enough $\epsilon>0$ such that the following perturbed $\pi^\circ$ is still a randomized test:
\[
\pi^\circ(\omega) + \epsilon[\bm\alpha_2(\omega) \ind_{\Omega_2^\circ}(\omega) - \bm\alpha_1(\omega) \ind_{\Omega_1^\circ}(\omega)] = 
\begin{cases}
1 - \epsilon\bm\alpha_1(\omega), & \omega \in \Omega_1^\circ, \\
\epsilon\bm\alpha_2(\omega), & \omega \in \Omega_2^\circ,
\end{cases}
\]
which means that the probability of accepting hypothesis $H_0$ is reduced on $\Omega_1^\circ$, and probability of accepting hypothesis $H_0$ is increased on $\Omega_2^\circ$. 
Recall $\bm\alpha = \bm\alpha_2 \ind_{\Omega_2^\circ} - \bm\alpha_1 \ind_{\Omega_1^\circ}$.
The optimality of $\pi^\circ$ implies that
\[
  \mathbb E_{P_1}[1-\pi^\circ(\omega) -\epsilon\bm\alpha(\omega)] + \mathbb E_{P_2}[\pi^\circ(\omega) + \epsilon\bm\alpha(\omega)] \geq \mathbb E_{P_1}[1-\pi^\circ(\omega) ] + \mathbb E_{P_2}[\pi^\circ(\omega)].
\] 
Dividing $\epsilon$ on both sides gives $\mathbb E_{P_1}[\bm\alpha(\omega)] - \mathbb E_{P_2}[\bm\alpha(\omega)] \leq 0$. 
% Note that we can relax the bounded constraint on functions $\bm\alpha_1$ and $\bm\alpha_2$ by letting $M\rightarrow \infty$. 
Moreover, the equality in \eqref{eq:radius_constr} holds by taking $\bm \alpha_1 = \bm\alpha_2 \equiv 0$, which proves \eqref{eq:radius_constr}. 

Next, we prove the sufficiency. 
Suppose \eqref{eq:radius_constr} holds.
For any randomized test $\pi$, set $\tilde{\bm\alpha}:= \pi - \pi^\circ$. Pick $\tilde{\bm\alpha}_1,\tilde{\bm\alpha}_2 \in \mathcal{B}_+(\Omega)$ such that
\[
  \tilde{\bm\alpha}_1(\omega)= \begin{cases} 1-\pi(\omega) & \text{if $\omega \in \Omega_1^\circ$,} \\
  0 & \textrm{otherwise;} \end{cases}\quad
  \tilde{\bm\alpha}_2(\omega)= \begin{cases} \pi(\omega) & \text{if $\omega \in \Omega_2^\circ$,} \\
  0 & \textrm{otherwise.}\end{cases}
\]
Then by the definition of $\pi^\circ$, we have $\tilde{\bm\alpha}(\omega) = \tilde{\bm\alpha}_2(\omega) \ind_{\Omega_2^\circ}(\omega) - \tilde{\bm\alpha}_1(\omega) \ind_{\Omega_1^\circ}(\omega)$ for all $\omega\in\Omega$.
It follows that $\mathbb E_{P_1}[\tilde{\bm\alpha}(\omega)] - \mathbb E_{P_2}[\tilde{\bm\alpha}(\omega)] \leq 0$, and consequently, 
\[
\begin{aligned}
  &\mathbb E_{P_1}[1-\pi(\omega)] + \mathbb E_{P_2}[\pi(\omega) ] \\
  = & \mathbb E_{P_1}[1-\pi^\circ(\omega) -\tilde{\bm\alpha}(\omega)] + \mathbb E_{P_2}[\pi^\circ(\omega) + \tilde{\bm\alpha}(\omega)] \\
  = & \mathbb E_{P_1}[1-\pi^\circ(\omega) ] + \mathbb E_{P_2}[\pi^\circ(\omega)] -  (\mathbb E_{P_1}[\tilde{\bm\alpha}(\omega)] - \mathbb E_{P_2}[\tilde{\bm\alpha}(\omega)] )\\
  \geq & \mathbb E_{P_1}[1-\pi^\circ(\omega) ] + \mathbb E_{P_2}[\pi^\circ(\omega)].
\end{aligned}
\]
This indicates that the risk of any test $\pi$ is greater than or equal to the risk of $\pi^\circ$, implying $\pi^\circ \in \argmin_{\pi} \Phi(\pi;P_1,P_2)$. Therefore we have completed the proof. 
\end{proof}

Let us proceed by defining the Lagrangian function
\begin{equation}\label{eq:lag}
\begin{aligned}
& L(P_1,P_2;\lambda_1,\lambda_2,\bm\alpha_1,\bm\alpha_2) \\
:=& \sum_{k=1}^2 \lambda_k\wass(P_k,Q_{k,n_k}) + \sum_{k=1}^2\sum_{j\neq k}  \left\{ \mathbb E_{P_k}[\bm\alpha_j(\omega) \ind_{\Omega_j^\circ}(\omega) - \bm\alpha_k(\omega) \ind_{\Omega_k^\circ}(\omega)] \right\},
\end{aligned}
\end{equation} 
where the second term is equivalent to $\int_\Omega [\bm\alpha_2(\omega) \ind_{\Omega_2^\circ}(\omega) - \bm\alpha_1(\omega) \ind_{\Omega_1^\circ}(\omega)] (dP_1 - dP_2)(\omega)$.
Using Lemma \ref{lem:lagrangian}, if $\pi^\circ \notin \argmin_{\pi} \Phi(\pi;P_1,P_2)$, then 
there exists functions $\bm\alpha_1',\bm\alpha_2' \in \mathcal B_+(\Omega)$ such that $\sum_{k=1}^2\sum_{j\neq k}\mathbb E_{P_k}[\bm\alpha_j'(\omega)  \ind_{\Omega_j^\circ}(\omega)  - \bm\alpha_k'(\omega)  \ind_{\Omega_k^\circ}(\omega) ] >0$, whence 
\[  
\sup_{\substack{ \bm\alpha_1,\bm\alpha_2 \in \mathcal B_+(\Omega)}} L(P_1,P_2;\lambda_1,\lambda_2,\bm\alpha_1,\bm\alpha_2)  \geq \lim_{t\to\infty} L(P_1,P_2;\lambda_1,\lambda_2,t\bm\alpha_1',t\bm\alpha_2') = +\infty.
\]
Therefore, we arrive at an equivalent formulation for the profile function $\profile_{n_1,n_2}$ defined in \eqref{eq:profile}:
\begin{equation}\label{eq:profile_lag}
\profile_{n_1,n_2}  = \inf_{\substack{P_1,P_2\in \scrP(\Omega) }} \sup_{\substack{\lambda_1,\lambda_2\geq 0 \\ \lambda_1+\lambda_2 \leq 1\\ {\bm\alpha_1,\bm\alpha_2 \in \mathcal B_+(\Omega)}}} L(P_1,P_2;\lambda_1,\lambda_2,\bm\alpha_1,\bm\alpha_2).
\end{equation}

In what follows, we prove the strong duality (i.e.~exchanging of $\sup$ and $\inf$) in five steps. 
We start by showing the weak duality and simplify the dual formulation of $\profile_{n_1,n_2}$. 
Next, we show that it suffices to restrict the feasible region of $\bm\alpha_1,\bm\alpha_2$ from $\cB_+(\Omega)$ to $\cB_+(\Omega)\cap\Lip(\Omega)$, which eventually leads to the set $\cA$ defined in \eqref{eq:setA}, and prove the strong duality by assuming the support $\Omega$ is compact.
Finally, we relax the compactness assumption.

\paragraph{\textbf{Step 1}} Weak duality.

Exchanging $\inf$ and $\sup$ in Equation \eqref{eq:profile_lag} yields
\begin{equation}\label{eq:weak_duality}
\begin{aligned}
\profile_{n_1,n_2}& \geq \sup_{\substack{\lambda_1,\lambda_2\geq 0 \\ \lambda_1+\lambda_2 \leq 1\\ \bm\alpha_1,\bm\alpha_2 \in \mathcal B_+(\Omega)}} \inf_{P_1,P_2\in\scrP(\Omega)}L(P_1,P_2;\lambda_1,\lambda_2,\bm\alpha_1,\bm\alpha_2).
\end{aligned}
\end{equation}
Let us simplify the right-hand side by deriving a closed-form solution to the inner $\inf$ problem. Recall that $\Gamma(P,Q)$ denotes the collection of all Borel probability measures on $\Omega\times\Omega$ with marginal distributions $P$ and $Q$. By the definition of Wasserstein metric, since the empirical distribution $Q_{k,n_k}$ is supported on a finite set $\hOmega_k=\{\homega_k^1,\ldots,\homega_k^{n_k}\}$ for $k=1,2$, we have
\[
\begin{aligned}
\lambda_k\wass(P_k,Q_{k,n_k}) =  
\inf_{\gamma_k \in \Gamma(P_k,Q_{k,n_k})  } \left\{ \sum_{i=1}^{n_k} \int_\Omega \lambda_kc(\omega,\homega_k^i) d\gamma_k(\omega, \homega_k^i)  \right\}.
\end{aligned}
\]
Moreover, for any distribution $\gamma_k \in \Gamma(P_k,Q_{k,n_k}) $, $k=1,2$, we have 
\[
\begin{multlined}
\sum_{j\neq k}  \left\{ \mathbb E_{P_k}[\bm\alpha_j(\omega)  \ind_{\Omega_j^\circ}(\omega)-\bm\alpha_k(\omega)  \ind_{\Omega_k^\circ}(\omega)] \right\}  \\
= \sum_{i=1}^{n_k} \int_\Omega \sum_{j\neq k}  [\bm\alpha_j(\omega) \ind_{\Omega_j^\circ}(\omega) - \bm\alpha_k(\omega) \ind_{\Omega_k^\circ}(\omega)] d\gamma_k(\omega, \homega_k^i) .
\end{multlined}
\]
Substituting the above equations to \eqref{eq:lag}, it follows that:
\[
\begin{aligned}
& L(P_1,P_2;\lambda_1,\lambda_2,\bm\alpha_1,\bm\alpha_2) \\
= & \sum_{k=1}^2 \inf_{\gamma_k \in \Gamma(P_k,Q_{k,n_k})  } \Big\{  \sum_{i=1}^{n_k}\int_\Omega \big[\lambda_kc(\omega,\homega_k^i)  \\
& \hspace{100pt} + \sum_{j\neq k}  (\bm\alpha_j(\omega)\ind_{\Omega_j^\circ}(\omega) - \bm\alpha_k(\omega)\ind_{\Omega_k^\circ}(\omega))   \big]d\gamma_k(\omega, \homega_k^i)  \Big\}.
\end{aligned}
\]
Thereby for fixed $\lambda_1$, $\lambda_2$, $\bm\alpha_1$, $\bm\alpha_2$, $\inf_{P_1,P_2} L(P_1,P_2;\lambda_1,\lambda_2,\bm\alpha_1,\bm\alpha_2)$ can be expressed equivalently as a minimization problem over $\gamma_k$, whose first marginal distribution can be arbitrary and second marginal is the empirical distribution $Q_{k,n_k}$, $k=1,2$: 
\[
\begin{aligned}
& \inf_{P_1,P_2} L(P_1,P_2;\lambda_1,\lambda_2,\bm\alpha_1,\bm\alpha_2) \\
= &  \sum_{k=1}^2  \inf_{\gamma_k \in \Gamma(\cdot,Q_{k,n_k}) } \Big\{ \sum_{i=1}^{n_k}   \int_\Omega \big[\lambda_kc(\omega,\homega_k^i)  \\
& \hspace{100pt} + \sum_{j\neq k}  (\bm\alpha_j(\omega)\ind_{\Omega_j^\circ}(\omega) - \bm\alpha_k(\omega)\ind_{\Omega_k^\circ}(\omega))   \big]d\gamma_k(\omega, \homega_k^i)  \Big\} \\
= & \sum_{k=1}^2 \frac{1}{n_k}  \sum_{i=1}^{n_k} \inf_{\omega\in\Omega} \big\{\lambda_k c( \omega,\homega_k^i) +\sum_{j\neq k}  (\bm\alpha_j(\omega)\ind_{\Omega_j^\circ}(\omega) - \bm\alpha_k(\omega)\ind_{\Omega_k^\circ}(\omega)) \big\},
\end{aligned}
\]
where $\Gamma(\cdot,Q_{k,n_k})$ denotes the collection of all Borel probability measures on $\Omega \times \Omega$ with second marginal being $Q_{k,n_k}$, and the last equality is attained by picking
\[
\gamma_k (\omega_k^i,\homega_k^i) = \frac{1}{n_k},\quad i=1,\ldots,n_k,\ k=1,2,
\]
where
\[
\omega_k^i \in \arg\min\limits_{\omega\in\Omega} \Big\{ \lambda_k c( \omega,\homega_k^i) + \sum\limits_{j\neq k}  (\bm\alpha_j(\omega)\ind_{\Omega_j^\circ}(\omega) - \bm\alpha_k(\omega)\ind_{\Omega_k^\circ}(\omega))  \Big\}.
\]
If the minimizer does not exist, we can argue similarly using a sequence of approximate minimizers. If there are multiple minimizers, we can simply choose one of them or distribute the probability mass $1/n_k$ uniformly on the optimal solution set. Therefore, we have the right-hand side of \eqref{eq:weak_duality} equals to 
\begin{equation}\label{eq:radius_dual}
\sup_{\substack{\lambda_1,\lambda_2\geq 0 \\ \lambda_1+\lambda_2 \leq 1\\ \bm\alpha_1,\bm\alpha_2 \in \mathcal B_+(\Omega)}} \!\!\! \sum_{k=1}^2 \E_{\homega_k\sim Q_{k,n_k}} \Big[\inf_{\omega\in\Omega} \Big\{ \lambda_k c( \omega,\homega_k) +\sum\limits_{j\neq k}[\bm\alpha_j(\omega)\ind_{\Omega_j^\circ}(\omega) - \bm\alpha_k(\omega)\ind_{\Omega_k^\circ}(\omega)] \Big\} \Big].
\end{equation}
In the sequel, we will refer to the right-hand side of \eqref{eq:radius_dual} as the dual problem.
 
\paragraph{\textbf{Step 2}} Restricting on the subset $\cA$ as defined in \eqref{eq:setA}.

We first prove that we can restrict $\bm\alpha_1$ and $\bm\alpha_2$ on the space of Lipschitz continuous functions without affecting the optimal value. 

For any feasible solution $(\lambda_1,\lambda_2, \bm\alpha_1,\bm\alpha_2)$ of the dual problem in \eqref{eq:radius_dual} such that the dual objective is finite, let us construct a modification $(\lambda_1,\lambda_2, \tilde{\bm\alpha}_1,\tilde{\bm\alpha}_2)$ which yields an objective value no worse than $(\lambda_1,\lambda_2, \bm\alpha_1,\bm\alpha_2)$, but enjoys a nicer continuity property. 
For $i=1,2,\ldots,n_1$, set
\[
\begin{aligned}
\phi_1(\homega_1^i) : &= \inf_{\omega\in\Omega}\{ \lambda_1 c( \omega,\homega_1^i) + \bm\alpha_2(\omega)\ind_{\Omega_2^\circ}(\omega) - \bm\alpha_1(\omega)\ind_{\Omega_1^\circ}(\omega)\}\\
& = \min \Big\{ \inf_{\omega\in \Omega_1^\circ} \{ \lambda_1 c( \omega,\homega_1^i) -\bm\alpha_1(\omega) \}, \inf_{\omega\in \Omega_2^\circ}\{ \lambda_1c( \omega,\homega_1^i) + \bm\alpha_2(\omega) \}  \Big\}.
\end{aligned}
\]
It follows that 
\begin{equation}\label{eq:alpha1_constr}
 \bm\alpha_1(\omega) \leq \lambda_1 c( \omega,\homega_1^i) -  \phi_1(\homega_1^i), \quad \forall \omega\in \Omega_1^\circ, \ \forall i=1,\ldots,n_1. 
\end{equation}
Define another function $\tilde{\bm\alpha}_1$ as
\begin{equation}\label{eq:ctransform1}
\tilde{\bm\alpha}_1(\omega) = \min_{i=1,\ldots,n_1}\left\{ \lambda_1 c( \omega,\homega_1^i) -  \phi_1(\homega_1^i) \right\}, \quad \forall \omega\in \Omega_1^\circ.
\end{equation}
This yields $\bm\alpha_1(\omega) \leq \tilde{\bm\alpha}_1(\omega)$ for all $\omega\in \Omega_1^\circ$, due to \eqref{eq:alpha1_constr}. Moreover, the objective value in \eqref{eq:radius_dual} associated with $(\lambda_1,\lambda_2,\tilde{\bm\alpha}_1,\bm\alpha_2)$ is no less than the value associated with $(\lambda_1,\lambda_2,\bm\alpha_1,\bm\alpha_2)$ since
\[
\begin{aligned}
\phi_1(\homega_1^i) & \leq \lambda_1 c( \omega,\homega_1^i) - \tilde{\bm\alpha}_1(\omega), \quad \forall \omega\in\Omega_1^\circ, \ \forall i=1,\ldots,n_1, \\
\lambda_2 c( \omega,\homega_2^j) + \bm\alpha_1(\omega) &\leq \lambda_2 c( \omega,\homega_2^j) + \tilde{\bm\alpha}_1(\omega) ,\quad \forall \omega\in\Omega_1^\circ, \ \forall j=i,\ldots,n_2. 
\end{aligned}
\]
Furthermore, the function $\tilde{\bm\alpha}_1$ defined in this way is Lipschitz with constant $\lambda_1$. Indeed, for any two points $\xi,\eta \in \Omega_1^\circ$, let $i_1$ and $i_2$ be the indices at which the minimum are attained in the definition \eqref{eq:ctransform1} for $\xi$ and $\eta$, respectively. We have
\[
\begin{aligned}
\tilde{\bm\alpha}_1(\xi) - \tilde{\bm\alpha}_1(\eta) & =[ \lambda_1 c( \xi,\homega_1^{i_1}) -  \phi_1(\homega_1^{i_1})] - [\lambda_1 c( \eta,\homega_1^{i_2}) -  \phi_1(\homega_1^{i_2})]
\\
& \leq [ \lambda_1 c( \xi,\homega_1^{i_2}) -  \phi_1(\homega_1^{i_2})] - [\lambda_1 c( \eta,\homega_1^{i_2}) -  \phi_1(\homega_1^{i_2})] \\
& = \lambda_1 [ c( \xi,\homega_1^{i_2}) - c( \eta,\homega_1^{i_2})]\\
& \leq  \lambda_1 c( \xi,\eta),
\end{aligned}
\]
where the last inequality is due to the triangle inequality of the metric $c(\cdot,\cdot)$; and the same inequality holds for $\tilde{\bm\alpha}_1(\eta) - \tilde{\bm\alpha}_1(\xi)$. 
In a similar fashion, for $j=1,2,\ldots,n_2$, define
\[
\begin{aligned}
\phi_2(\homega_2^j) : &= \inf_{\omega}\{ \lambda_2 c( \omega,\homega_2^j) +\tilde{\bm\alpha}_1(\omega)\ind_{\Omega_1^\circ}(\omega) - \bm\alpha_2(\omega)\ind_{\Omega_2^\circ}(\omega)\}\\
& = \min \Big\{ \inf_{\omega\in \Omega_1^\circ} \{ \lambda_2 c( \omega,\homega_2^j)  + \tilde{\bm\alpha}_1(\omega) \}, \inf_{\omega\in \Omega_2^\circ}\{ \lambda_2 c( \omega,\homega_2^j)  - \bm\alpha_2(\omega) \}  \Big\},
\end{aligned}
\]
and set 
\begin{equation}\label{eq:ctransform2}
\tilde{\bm\alpha}_2(\omega) := \min_{j=1,\ldots,n_2}\left\{ \lambda_2 c( \omega,\homega_2^j) - \phi_2(\homega_2^j) \right\}, \quad \forall \omega\in \Omega_2^\circ.
\end{equation}
Then $\bm\alpha_2(\omega) \leq \tilde{\bm\alpha}_2(\omega)$ for all $\omega\in \Omega_2^\circ$ and the objective value associated with $(\lambda_1,\lambda_2,\tilde{\bm\alpha}_1,\tilde{\bm\alpha}_2)$ is no less than the objective value associated with $(\lambda_1,\lambda_2,\tilde{\bm\alpha}_1,\bm\alpha_2)$; and $\tilde{\bm\alpha}_2$ is Lipschitz with constant $\lambda_2$. 
Since we are in the region $\{\lambda_1,\lambda_2\geq 0, \lambda_1+\lambda_2\leq 1\}$, the argument above proves that without loss of generality we can restrict $\bm\alpha_1,\bm\alpha_2$ on the set of $1$-Lipschitz continuous functions.

Observe that the objective value does not change if we shift $\bm\alpha_k$ by any constant $C_k$, $k=1,2$. Hence, without loss of generality, we can only consider those satisfying $\bm\alpha_k(\omega_k^\circ)=0$ without affecting the optimal value, where $\omega_k^\circ\in\Omega_k^\circ$, $k=1,2$.
By the above argument, we have shown that it suffices to restrict the feasible region on $\cA$.

\paragraph{\textbf{Step 3}} Strong duality for compact space. 
  
Now assume $\Omega$ is compact. We aim to prove the strong duality by applying Sion's minimax theorem to the Lagrangian $L(P_1,P_2;\lambda_1,\lambda_2,\bm\alpha_1,\bm\alpha_2)$ defined in \eqref{eq:lag}.
%Note that the continuity of the functions $\bm\alpha_2(\omega) \ind_{\Omega_2^\circ}(\omega) - \bm\alpha_1(\omega) \ind_{\Omega_1^\circ}(\omega)$ is a critical issue. 
Observe that $L(P_1,P_2;\lambda_1,\lambda_2,\bm\alpha_1,\bm\alpha_2)$ is convex in $P_k$, linear in $\lambda_k$ and $\bm\alpha_k$; by Prokhorov's theorem \cite{prokhorov1956convergence}, the convex space $\calP(\Omega)\times \calP(\Omega)$ is compact since $\Omega$ is relatively compact with respect to the weak topology; the space $\{\lambda_1,\lambda_2\geq 0, \lambda_1+\lambda_2\leq 1\}$ is also a convex compact space. 
The feasible region of $\bm\alpha_k,k=1,2$ belongs to a linear topological space under the sup-norm. This justifies the conditions for Sion's minimax theorem, thereby we can exchange sup and inf in \eqref{eq:lag} when $\Omega$ is compact.

\paragraph{\textbf{Step 4}} Relaxing the compactness assumption when the cost is bounded.
   
We now relax the compactness assumption made in the previous step, using a technique similar to the proof of Theorem 1.3 in \cite{villani2003topics}. 
We temporarily assume the cost function $c(\cdot,\cdot)$ is bounded by a positive constant $C$ and is uniformly continuous. We will relax the bounded assumption later.
We already have the weak duality:
\[
\begin{aligned} 
& v_1 :=\inf_{\substack{P_1,P_2\in \scrP(\Omega) }} \sup_{\substack{\lambda_1,\lambda_2\geq 0 \\ \lambda_1+\lambda_2 \leq 1\\ \bm\alpha_1,\bm\alpha_2 \in \mathcal B_+(\Omega)}} L(P_1,P_2;\lambda_1,\lambda_2,\bm\alpha_1,\bm\alpha_2)  \\ 
&\geq \sup_{\substack{\lambda_1,\lambda_2\geq 0 \\ \lambda_1+\lambda_2 \leq 1\\ \bm\alpha_1,\bm\alpha_2 \in \mathcal B_+(\Omega)}}  \!\!\!\!\sum_{k=1}^2 \E_{\homega_k\sim Q_{k,n_k}}  \Big[\inf_{\omega \in \Omega} \Big\{ \lambda_k c( \omega,\homega_k) +\sum\limits_{j\neq k}  [\bm\alpha_j(\omega)\ind_{\Omega_j^\circ}(\omega) - \bm\alpha_k(\omega)\ind_{\Omega_k^\circ}(\omega)] \Big\} \Big]\\
&=: v_2.
\end{aligned}
\]
In the following we show that $v_1 \leq v_2$. 

For any $\epsilon>0$, let $\Omega^\epsilon \subset \Omega$ be a compact subset sufficiently large, such that $P_k^\circ(\Omega \setminus \Omega^{\epsilon}) \leq \epsilon$ and $Q_{k,n_k}(\Omega^{\epsilon})=1$, $k=1,2$. This is always possible since $Q_{k,n_k}$ is the empirical distribution and with finite support.
Then the previous steps imply that the strong duality holds on $\Omega^\epsilon$:
\[
\begin{aligned} 
&v_1^\epsilon := \inf_{\substack{P_1,P_2\in \scrP(\Omega^\epsilon) }} \sup_{\substack{\lambda_1,\lambda_2\geq 0 \\ \lambda_1+\lambda_2 \leq 1\\ \bm\alpha_1,\bm\alpha_2 \in \mathcal B_+(\Omega^\epsilon)}} L(P_1,P_2;\lambda_1,\lambda_2,\bm\alpha_1,\bm\alpha_2)  \\ 
&= \sup_{\substack{\lambda_1,\lambda_2\geq 0 \\ \lambda_1+\lambda_2 \leq 1\\ \bm\alpha_1,\bm\alpha_2 \in \mathcal B_+(\Omega^\epsilon)}}  \!\!\!\!\sum_{k=1}^2 \E_{\homega_k\sim Q_{k,n_k}} \Big[\inf_{\omega \in \Omega^\epsilon} \Big\{ \lambda_k c( \omega,\homega_k) +\sum\limits_{j\neq k}  [\bm\alpha_j(\omega)\ind_{\Omega_j^\circ}(\omega) - \bm\alpha_k(\omega)\ind_{\Omega_k^\circ}(\omega)] \Big\} \Big]\\
&=: v_2^\epsilon.
\end{aligned}
\]

Consider the $\inf\sup$ problem defining $v_1$.
For the optimal solution $(P_1^\epsilon,P_2^\epsilon)$ to the $\inf\sup$ problem that induces $v_1^\epsilon$, we define distributions $\tilde{P}_1,\tilde{P}_2$ via
\[
\tilde P_k(A) =  P_k^\circ(\Omega^{\epsilon}) \cdot P_k^\epsilon(A\cap\Omega^\epsilon) + P_k^\circ(A\cap(\Omega \setminus \Omega^{\epsilon})),\quad  \forall \textnormal{ Borel set }A\subset\Omega.
\]
Recall $\bm\alpha = \bm\alpha_2 \ind_{\Omega_2^\circ} - \bm\alpha_1 \ind_{\Omega_1^\circ}$. 
We compare the Lagrangian function $L$ defined in \eqref{eq:lag} associated with $(\tilde P_1,\tilde P_2)$ and $(P_1^\epsilon,P_2^\epsilon)$. For the first term in \eqref{eq:lag}, we have that
\[
\wass(\tilde P_k,Q_{k,n_k}) \leq P_k^\circ(\Omega^{\epsilon}) \wass(P_k^\epsilon,Q_{k,n_k}) + CP_k^\circ(\Omega \setminus \Omega^{\epsilon}) \leq \wass(P_k^\epsilon,Q_{k,n_k}) + C\epsilon.
\] 
For the second term in \eqref{eq:lag}, we have
\[
\begin{aligned}
&\int_\Omega \bm\alpha(\omega)\, (\tilde{P}_1 - \tilde{P}_2)(d\omega) \\
= & \int_{\Omega^\epsilon} \bm\alpha(\omega)\, (P_1^\circ(\Omega^\epsilon) P_1^\epsilon - P_2^\circ(\Omega^\epsilon) P_2^\epsilon)(d\omega) + \int_{\Omega \setminus \Omega^\epsilon} \bm\alpha(\omega)\, (P_1^\circ - P_2^\circ)(d\omega). 
\end{aligned}
\]
By definition of $\Omega_1^\circ, \Omega_2^\circ$, we have $\int_{\Omega \setminus \Omega^\epsilon} \bm\alpha(\omega)(P_1^\circ - P_2^\circ)(d\omega) \leq 0$.
Moreover,
\[
\begin{aligned}
& \int_{\Omega^\epsilon} \bm\alpha(\omega)\, (P_1^\circ(\Omega^\epsilon) P_1^\epsilon - P_2^\circ(\Omega^\epsilon) P_2^\epsilon)(d\omega) \\
= & \begin{cases} P_1^\circ(\Omega^\epsilon) \int_{\Omega^\epsilon} \bm\alpha(\omega) ( P_1^\epsilon - P_2^\epsilon)(d\omega) - (P_2^\circ(\Omega^\epsilon) - P_1^\circ(\Omega^\epsilon)) \int_{\Omega^\epsilon} \bm\alpha(\omega) P_2^\epsilon(d\omega), &\\  \hspace{240pt} \text{if }P_1^\circ(\Omega^\epsilon) \leq P_2^\circ(\Omega^\epsilon); \\
P_2^\circ(\Omega^\epsilon) \int_{\Omega^\epsilon} \bm\alpha(\omega) ( P_1^\epsilon - P_2^\epsilon)(d\omega) + (P_1^\circ(\Omega^\epsilon) - P_2^\circ(\Omega^\epsilon)) \int_{\Omega^\epsilon} \bm\alpha(\omega) P_1^\epsilon(d\omega), & \\ \hspace{240pt} \text{if } P_1^\circ(\Omega^\epsilon) > P_2^\circ(\Omega^\epsilon).
\end{cases}
\end{aligned}
\]
By definition $\int_{\Omega^\epsilon} \bm\alpha(\omega) ( P_1^\epsilon - P_2^\epsilon)(d\omega)\leq 0$ and $ P_k^\circ(\Omega^\epsilon)\geq 1-\epsilon$, thereby
\[
P_k^\circ(\Omega^\epsilon) \int_{\Omega^\epsilon} \bm\alpha(\omega) ( P_1^\epsilon - P_2^\epsilon)(d\omega)\leq (1-\epsilon) \int_{\Omega^\epsilon} \bm\alpha(\omega) ( dP_1^\epsilon - dP_2^\epsilon)(\omega) \leq 0.
\]
Moreover, since $P_k^\circ(\Omega^\epsilon)\geq 1-\epsilon$, we have $|P_1^\circ(\Omega^\epsilon) - P_2^\circ(\Omega^\epsilon)| \leq \epsilon$,  consequently we have
\[
|P_1^\circ(\Omega^\epsilon) - P_2^\circ(\Omega^\epsilon)| \int_{\Omega^\epsilon} \bm\alpha(\omega) dP_k^\epsilon(\omega) \leq \epsilon \int_\Omega c(\omega,\omega_k^0) dP_k^\epsilon(\omega) \leq C\epsilon,
\]
where the last inequality is due to the $1$-Lipschitz property of $\bm\alpha_k$ and $C$ may be a different constant. 
%   It is always possible to select $\alpha(\xi)\leq C$ and it would simplify the above inequality to 
%   \[
%   |P_1^\circ(\Omega^\epsilon) - P_2^\circ(\Omega^\epsilon)| \int_{\Omega^\epsilon} \bm\alpha(\omega) dP_k^\epsilon(\omega) \leq \epsilon C.
%   \]
Combining with previous inequality that $\wass(\tilde P_k,Q_{k,n_k}) \leq \wass(P_k^\epsilon,Q_{k,n_k}) + C\epsilon$, $k =1,2$, we have 
\[
v_1 \leq v_1^\epsilon + 2C\epsilon .
\]

Now consider the dual problem defining $v_2$. Let $(\bm\alpha_1^\epsilon,\bm\alpha_2^\epsilon)$ be the optimal solution to the dual problem supported on the subset $\Omega^{\epsilon}$. 
We will construct an approximate maximizer $(\tilde{\bm\alpha}_1,\tilde{\bm\alpha}_2)$ of the original dual problem
from $(\bm\alpha_1^\epsilon,\bm\alpha_2^\epsilon)$. 
To this end, let us define
\[
\begin{aligned}
\phi_1^\epsilon(\homega_1^i) :&= \min \Big\{ \inf_{\omega\in \Omega_1^\circ\cap\Omega^{\epsilon}} \{ \lambda_1 c( \omega,\homega_1^i)  - \bm\alpha_1^\epsilon(\omega) \}, \inf_{\omega\in \Omega_2^\circ\cap\Omega^{\epsilon}}\{ \lambda_1 c( \omega,\homega_1^i)  + \bm\alpha_2^\epsilon(\omega) \}  \Big\},\\
\phi_2^\epsilon(\homega_2^j) : &= \min \Big\{ \inf_{\omega\in \Omega_1^\circ\cap\Omega^{\epsilon}} \{ \lambda_2 c( \omega,\homega_2^j)  + \bm\alpha_1^\epsilon(\omega) \}, \inf_{\omega\in \Omega_2^\circ\cap\Omega^{\epsilon}}\{ \lambda_2 c( \omega,\homega_2^j)  - \bm\alpha_2^\epsilon(\omega) \}  \Big\}.
\end{aligned}
\]
From the above equations we have that $\bm\alpha_1^\epsilon,\bm\alpha_2^\epsilon$ satisfy:
\begin{equation}\label{eq:alpha_eps_cond}
\begin{aligned}
\bm\alpha_1^\epsilon(\omega) & \leq \lambda_1 c( \omega,\homega_1^i) - \phi_1^\epsilon(\homega_1^i),\quad \forall \omega \in \Omega^\epsilon, \ i =1,\ldots,n_1, \\
\bm\alpha_2^\epsilon(\omega) & \leq \lambda_2 c( \omega,\homega_2^j) - \phi_2^\epsilon(\homega_2^j),\quad \forall \omega \in \Omega^\epsilon, \ j =1,\ldots,n_2.
\end{aligned}    
\end{equation}
Define $\tilde{\bm\alpha}_1,\tilde{\bm\alpha}_2 $ as
\begin{equation}\label{eq:tilde_alpha_cond}
\begin{aligned}
\tilde{\bm\alpha}_1(\omega) &= \min_{1\leq i \leq n_1} \{ \lambda_1 c(\omega,\homega_1^i) - \phi_1^\epsilon(\homega_1^i)\}, \quad \forall \omega \in \Omega, \\
\tilde{\bm\alpha}_2(\omega) &= \min_{1\leq j \leq n_2} \{ \lambda_2 c(\omega,\homega_2^j) - \phi_2^\epsilon(\homega_2^j)\}, \quad \forall \omega \in \Omega.
\end{aligned}
\end{equation}
This implies that $\phi_k^\epsilon(\homega_k^i) \leq \inf_{\omega\in \Omega_k^\circ} \{ \lambda_k c( \omega,\homega_k^i)  - \tilde{\bm\alpha}_k(\omega) \}$, $k=1,2$, $i=1,\ldots,n_k$. 
Comparing \eqref{eq:tilde_alpha_cond} and \eqref{eq:alpha_eps_cond}, we have that $\tilde{\bm\alpha}_k(\omega) \geq \bm\alpha_k^\epsilon(\omega)$, $k=1,2$, for $\omega\in\Omega^\epsilon$. Consequently, we have  
\[
\begin{aligned}
\phi_1^\epsilon(\homega_1^i) &\leq \min \Big\{ \inf_{\omega\in \Omega_1^\circ\cap\Omega^{\epsilon}} \{ \lambda_1 c( \omega,\homega_1^i)  - \tilde{\bm\alpha}_1(\omega) \}, \inf_{\omega\in \Omega_2^\circ\cap\Omega^{\epsilon}}\{ \lambda_1 c( \omega,\homega_1^i)  + \tilde{\bm\alpha}_2(\omega) \}  \Big\},\\
\phi_2^\epsilon(\homega_2^j) &\leq \inf \Big\{ \min_{\omega\in \Omega_1^\circ\cap\Omega^{\epsilon}} \{ \lambda_2 c( \omega,\homega_2^j)  + \tilde{\bm\alpha}_1(\omega) \}, \inf_{\omega\in \Omega_2^\circ\cap\Omega^{\epsilon}}\{ \lambda_2 c( \omega,\homega_2^j)  - \tilde{\bm\alpha}_2(\omega) \}  \Big\}.
\end{aligned}
\]
Moreover, we can choose $\Omega^\epsilon$ sufficiently large so that for every $\omega\in \Omega_2^\circ \cap (\Omega\setminus \Omega^\epsilon)$,
\[
\begin{aligned}
\lambda_1 c( \omega,\homega_1^i) + \tilde{\bm\alpha}_2(\omega) & = \lambda_1 c( \omega,\homega_1^i) + \lambda_2 c(\omega,\homega_2^j) - \phi_2^\epsilon(\homega_2^j)\\
& \geq \inf_{\omega\in \Omega_2^\circ\cap\Omega^{\epsilon}}\{ \lambda_1 c( \omega,\homega_1^i)  + \bm\alpha_2^\epsilon(\omega) \}
\geq \phi_1^\epsilon(\homega_1^i),
\end{aligned}
\]
where $j$ is the minimizer in the definition \eqref{eq:tilde_alpha_cond}. 
%  Since $\Omega^\epsilon$ is a convex compact set and $\homega_1^i,\homega_2^j\in\Omega^\epsilon$, there exists a point $\omega^\epsilon\in\Omega_2^\circ\cap\Omega^\epsilon$ such that $\lambda_1 c( \omega,\homega_1^i) + \lambda_2 c(\omega,\homega_2^j) \geq \lambda_1 c( \omega^\epsilon,\homega_1^i) + \lambda_2 c(\omega^\epsilon,\homega_2^j)$, thereby
%   \[
%   \begin{aligned}
%   \lambda_1 c( \omega,\homega_1^i) + \tilde{\bm\alpha}_2(\omega) & \geq \lambda_1 c( \omega^\epsilon,\homega_1^i) + \lambda_2 c(\omega^\epsilon,\homega_2^j) - \phi_2^\epsilon(\homega_2^j) \\
%     & \geq \min_{\omega\in \Omega_2^\circ\cap\Omega^{\epsilon}}\{ \lambda_1 c( \omega,\homega_1^i)  + \bm\alpha_2^\epsilon(\omega) \}\\
%   & \geq \phi_1^\epsilon(\homega_1^i),
%   \end{aligned}
%   \]
%   and similarly for $\omega\in \Omega_1^\circ \cap (\Omega\setminus \Omega^\epsilon)$, we have
%   \[
%   \lambda_2 c( \omega,\homega_2^j) + \tilde{\bm\alpha}_1(\omega) \geq \phi_2^\epsilon(\homega_2^j).
%   \]
Combining these together, we have 
\[
\begin{aligned}
\phi_1^\epsilon(\homega_1^i) &\leq \min \Big\{ \inf_{\omega\in \Omega_1^\circ} \{ \lambda_1 c( \omega,\homega_1^i)  - \tilde{\bm\alpha}_1(\omega) \}, \inf_{\omega\in \Omega_2^\circ}\{ \lambda_1 c( \omega,\homega_1^i)  + \tilde{\bm\alpha}_2(\omega) \}  \Big\},\\
\phi_2^\epsilon(\homega_2^j) &\leq \min \Big\{ \inf_{\omega\in \Omega_1^\circ} \{ \lambda_2 c( \omega,\homega_2^j)  + \tilde{\bm\alpha}_1(\omega) \}, \inf_{\omega\in \Omega_2^\circ}\{ \lambda_2 c( \omega,\homega_2^j)  - \tilde{\bm\alpha}_2(\omega) \}  \Big\}.
\end{aligned}
\]
Therefore, from $\tilde{\bm\alpha}_1,\tilde{\bm\alpha}_2$  defined in \eqref{eq:tilde_alpha_cond}, we see $v_2 \geq v_2^\epsilon$. Combine with previous argument, we have
\[
v_2^\epsilon \leq v_2 \leq v_1 \leq v_1^\epsilon + 2C\epsilon.
\]
By letting $\epsilon\rightarrow 0$, we have shown the strong duality, provided that the cost function is bounded.

\paragraph{\textbf{Step 5}} Relaxing the bounded cost assumption.

Next, we turn to the general case with cost function by writing $c:=\sup_m c_m$, where $c_m(x,y) = \min\{c(x,y),m\}$ is the truncated cost function that are bounded for each $m\in \mathbb N$. Let $v_1^m$ be the optimal value of the primal problem under cost $c_m$, and $v_2^m$ denote the optimal value of the dual problem under cost $c_m$. More specifically, let 
\[\label{eq:lag_m}
\begin{aligned}
& L^m(P_1,P_2;\lambda_1,\lambda_2,\bm\alpha_1,\bm\alpha_2) \\
:=& \sum_{k=1}^2 \lambda_k\wass^m(P_k,Q_{k,n_k}) + \sum_{k=1}^2\sum_{j\neq k}  \left\{ \mathbb E_{P_k}[\bm\alpha_j(\omega) \ind_{\Omega_j^\circ}(\omega) - \bm\alpha_k(\omega) \ind_{\Omega_k^\circ}(\omega)] \right\},
\end{aligned}
\]
where $\wass^m(P_k,Q_{k,n_k})$ is the Wasserstein distance associated with cost function $c_m(\cdot,\cdot)$.
Define
\[
\begin{aligned} 
&v_1^m := \inf_{\substack{P_1,P_2\in \scrP(\Omega) }} \sup_{\substack{\lambda_1,\lambda_2\geq 0 \\ \lambda_1+\lambda_2 \leq 1\\ \bm\alpha_1,\bm\alpha_2 \in \mathcal B_+(\Omega)}} L^m(P_1,P_2;\lambda_1,\lambda_2,\bm\alpha_1,\bm\alpha_2)  \\ 
&= \sup_{\substack{\lambda_1,\lambda_2\geq 0 \\ \lambda_1+\lambda_2 \leq 1\\ \bm\alpha_1,\bm\alpha_2 \in \mathcal B_+(\Omega)}}  \!\!\!\!\sum_{k=1}^2 \E_{\homega_k\sim Q_{k,n_k}} \Big[\inf_{\omega \in \Omega} \Big\{ \lambda_k c_m( \omega,\homega_k) +\sum\limits_{j\neq k}  [\bm\alpha_j(\omega)\ind_{\Omega_j^\circ}(\omega) - \bm\alpha_k(\omega)\ind_{\Omega_k^\circ}(\omega)] \Big\} \Big]\\
&=: v_2^m.
\end{aligned}
\]
We have proved $v_1^m = v_2^m$ in previous steps. And clearly we have $v_2^m \leq v_2$ since $c_m \leq c$, leading to $v_1^m = v_2^m \leq v_2 \leq v_1$, so we only need to show $v_1 = \sup_m v_1^m$. 

Observe that $\wass^m(P_k,Q_{k,n_k})$ is a non-decreasing sequence bounded above by $\wass(P_k,Q_{k,n_k})$. If $\{(P_{1,l}^m,P_{2,l}^m)\}_{l\in \mathbb N}$ is a minimizing sequence for the problem $v_1^m$, then we can extract a subsequence that converges weakly to some probability measure $P_1^m,P_2^m$ \cite{villani2003topics}. %\rgao{Reference}

We claim that the sequence $\{P_k^m\}_{m\in \mathbb N}$ is relatively compact with respect to the weak topology, $k=1,2$. 
% This can be proved by contradiction, using a similar argument as in \cite{si2020quantifying}. 
To show this, suppose $\{P_k^m\}_{m\in \mathbb N}$ is not relatively compact, then there exists $\epsilon>0$ such that for any compact set $A$ and any $m_0 \in \mathbb{N}$, there exists $m > m_0$ such that $P_k^m (A) \geq \epsilon$. 
We choose $m_0 = \lceil\wass(Q_{k,n_k},P_k^\circ)/\epsilon\rceil$ and a set $A$ such that $\inf_{\omega\in A, \homega \in \hOmega} c(\omega,\homega) \geq m_0$. 
Then for any $m>m_0$, we have 
\[
\begin{aligned}
\wass^m(Q_{k,n_k},P_k^m) &= \min_{\gamma\in\Gamma(P_k^m,Q_{k,n_k})} \left\{\E_{(\omega,\omega')\sim\gamma} \left[c_m(\omega,\omega')\right]  \right\}\\
&> m_0P_k^m (A) \geq m_0\epsilon \geq \wass(Q_{k,n_k},P_k^\circ),
\end{aligned}
\] 
while at the same time we have 
\[
\wass^m(Q_{k,n_k},P_1^m) \leq \wass^m(Q_{k,n_k},P_1^\circ) \leq \wass(Q_{k,n_k},P_1^\circ),
\]
which is a contradiction. Therefore $\{P_k^m\}_{m\in \mathbb N}$ is relatively compact and we can extract a subsequence that converges to some probability measure $P_k^\ast$.   

For any $m_1 > m_2$, we have $\wass^{m_1}(P_k^{m_1},Q_{k,n_k}) \geq \wass^{m_2}(P_k^{m_1},Q_{k,n_k})$, and 
\[
\limsup_{m_1\rightarrow\infty}\wass^{m_1}(P_k^{m_1},Q_{k,n_k}) \geq \limsup_{m_1\rightarrow\infty}\wass^{m_2}(P_k^{m_1},Q_{k,n_k}) \geq \wass^{m_2}(P_k^\ast,Q_{k,n_k}).
\]
Moreover, $\wass^{m_2}(P_k^\ast,Q_{k,n_k})$ is a non-decreasing sequence and converges to $\wass(P_k^\ast,Q_{k,n_k})$ as $m_2\rightarrow\infty$, hence:
\[
\limsup_{m\rightarrow\infty} v_1^m = \limsup_{m\rightarrow\infty} \wass^{m}(P_k^{m},Q_{k,n_k}) \geq \wass(P_k^\ast,Q_{k,n_k})=v_1.
\]
%   Therefore we have $
%   \lim_{m\rightarrow \infty} \wass_m(P_k^m,Q_{k,n_k}) = \lim_{m\rightarrow \infty} \wass_m(P_k^*,Q_{k,n_k})$. 
% Therefore, by letting $m\rightarrow \infty$ we have $v_1 = v_2$ for general space $\Omega$ and general cost function $c$. 
Thereby we complete the proof. 
\qed

% the equi-continuous may be related to holder? for p=1, lipschitz,for q<1 -> Holder?
% from bounded to unbounded, the key idea is to bound the growth rate< C, and the expectation under the true distribution for C is finite, therefore we can focus on the compact subset...

\subsection{Proof of Theorem~\ref{thm:radius}} \label{proof:radius}

Our analysis starts from the observation that $$\profile_{n_1,n_2} \leq \sup_{\bm\alpha\in\cA} G_{n_1,n_2}(\bm\alpha).
$$
First recall that $\cA$ is defined in \eqref{eq:setA} as:
\[
\cA = \Big\{\bm\alpha=\bm\alpha_2 \ind_{\Omega_2^\circ} - \bm\alpha_1 \ind_{\Omega_1^\circ}:\;\bm\alpha_k\in\cB_+(\Omega_k^\circ)\cap\Lip(\Omega_k^\circ),\,\bm\alpha(\omega_k^\circ)=0,\,k=1,2 \Big\}. 
\]
We then provide an upper bound on $\sup_{\bm\alpha\in\cA} G_{n_1,n_2}(\bm\alpha)$ detailed as follows.
By definition of $\cA$, $\bm\alpha(\homega_1^i)\leq0$ for $\homega_1^i\in\Omega_1^\circ$ and $\bm\alpha(\homega_2^j)\geq0$ for $\homega_2^j\in\Omega_2^\circ$. Therefore, to maximize $G_{n_1,n_2}(\bm\alpha)$, we should let $\bm\alpha(\homega_1^i)=0$ for $\homega_1^i\in\Omega_1^\circ$ and $\bm\alpha(\homega_2^j)=0$ for $\homega_2^j\in\Omega_2^\circ$. 
% Indeed, if for $\homega_1^i\in\Omega_1^\circ$, we let $\bm\alpha(\homega_1^i)=-\Delta$ for any constant $\Delta>0$, then $G_N(\bm\alpha)$ attains its maximum by setting $\bm\alpha(\homega_2^j)=-\Delta-\min_{\homega_1^i \in\Omega_1^\circ}c(\homega_2^j,\homega_1^i)$ for $\homega_2^j \in \Omega_1^\circ$. By the definition of the set $\Omega_1^\circ$, we have for any $\epsilon>0$, there exists $N_0$ such that for any $N>N_0$, we have
% \[
% \P[|\{\homega_1^i: \homega_1^i\in\Omega_1^\circ\}| > |\{\homega_2^j: \homega_2^j\in\Omega_1^\circ\}|] \geq 1-\epsilon,
% \]
% which implies that with high probability the objective value of $G_N(\bm\alpha)$ by setting $\bm\alpha(\homega_1^i)=-\Delta<0$ is smaller than the value obtained by setting $\bm\alpha(\homega_1^i)=0$, and similar conclusion holds for $\bm\alpha(\homega_2^j)$. 
%
% Now let $\bm\alpha(\homega_1^i)=0$ for $\homega_1^i\in\Omega_1^\circ$ and $\bm\alpha(\homega_2^j)=0$ for $\homega_2^j\in\Omega_2^\circ$. 
In addition, since $ \bm\alpha_1,\bm\alpha_2$ are $1$-Lipschitz, we have $\bm\alpha(\homega_1^i) \leq \min_{j:\;\homega_2^j \in\Omega_2^\circ}c(\homega_1^i,\homega_2^j)$ for $\homega_1^i\in\Omega_2^\circ$ and $\bm\alpha(\homega_2^j) \geq -\min_{i:\;\homega_1^i \in\Omega_1^\circ}c(\homega_2^j,\homega_1^i)$ for $\homega_2^j\in\Omega_1^\circ$. 
Hence we have
%   \[
%   \sup_{\bm\alpha\in\cA} G_N(\bm\alpha) = \frac{1}{n_1(N)}\sum_{\homega_1^i\notin\Omega_1^\circ}c(\homega_1^i, \Omega_2^\circ) + \frac{1}{n_2(N)}\sum_{\homega_2^j\notin\Omega_2^\circ}c(\homega_2^j, \Omega_1^\circ),
%   \]
%   \[
%   \begin{multlined}
%   \sup_{\bm\alpha\in\cA} G_N(\bm\alpha) = \frac{1}{n_1(N)}\sum_{\homega_1^i\notin\Omega_1^\circ}\min\{c(\homega_1^i, \Omega_2^\circ), \min_{\homega_2^j \in\Omega_2^\circ}c(\homega_1^i,\homega_2^j)\} \\
%   + \frac{1}{n_2(N)}\sum_{\homega_2^j\notin\Omega_2^\circ} \min\{ c(\homega_2^j, \Omega_1^\circ), \min_{\homega_1^i \in\Omega_1^\circ}c(\homega_2^j,\homega_1^i) \},
%   \end{multlined}
%   \]
%which can be shown to be equivalent to
\begin{equation}\label{eq:bound_dist}
\sup_{\bm\alpha\in\cA} G_{n_1,n_2}(\bm\alpha) = \frac{1}{n_1} \sum_{i:\;\homega_1^i\in\Omega_2^\circ}\min_{j:\;\homega_2^j \in\Omega_2^\circ}c(\homega_1^i,\homega_2^j) + \frac{1}{n_2} \sum_{j:\;\homega_2^j\in\Omega_1^\circ} \min_{i:\;\homega_1^i \in\Omega_1^\circ}c(\homega_2^j,\homega_1^i).    
\end{equation}
Note that the summation equals $0$ if there is no point $\homega_1^i$ falling into the set $\Omega_2^\circ$ and no point $\homega_2^j$ falling into the set $\Omega_1^\circ$. 
% This means that when the true distributions $P_1^\circ$ and $P_2^\circ$ have no common support, then we always have $\sup_{\bm\alpha\in\cA} G_N(\bm\alpha)=0$, meaning that such scenario is the easiest to discriminate and setting the radius to be 0 will be enough.

In light of above, in order to obtain the asymptotic upper bound on the profile function $\profile_{n_1,n_2}$, we study the right-hand side of the above inequality instead, which is relatively simpler since it only involves the minimum distance type statistics of two sample sets. 
We state the following lemma, whose proof is adapted from the asymptotic moments of near-neighbour distance distributions in \cite{penrose2011laws,wade2007explicit,penrose2003weak}.

\begin{lemma}\label{lem:min_dist_unbound}
Let $\{x_1,\ldots,x_n\}$ be a set of points selected independently at random from $\Omega\subset\R^d$ according to the sampling distribution $F$ with density function $f$. Let $x$ be a random variable sampled from the distribution $G$ with density function $g$. Suppose that $f(x)$, $g(x)$ are densities with
\[
\int_{\Omega} g(x)f(x)^{-1/d}dx < \infty,
\]
and for some $\epsilon>0$ we have $\sup_{n \in \mathbb{N}} \E_{x\sim f, x_1,\ldots,x_n\sim g}[(n^{1/d}\min_{1\leq i\leq n}  \left\Vert x - x_i\right\Vert)^{1+\epsilon}]<\infty$. 
% \[
% \int_{\Omega} \left\Vert x \right\Vert^r f(x)dx < \infty,
% \]
% for some $r> md/(d-m)$. 
Then
\[
n^{1/d}\E\{\min_{1\le i \le n} \left\Vert x-x_i\right\Vert\} \rightarrow \frac{\Gamma(1+1/d)}{V_d^{1/d}} \int_{\R^d} g(x) f(x)^{-1/d}dx,
\]
as $n\rightarrow \infty$, where $V_d = \pi^{d/2}/\Gamma(1+d/2)$ is the volume of the unit ball in $\R^d$. 
\end{lemma}

% In particular, if the support $\Omega$ is a compact convex set with $\mu(\Omega)=1$ under the Lebesgue measure $\mu$ on $\R^d$ and the density $f$ is continuous on $\Omega$ and has bounded partial derivatives on $\Omega$, and $f(\omega)> 0$ for all $\omega\in \Omega$, then we have further the rate of convergence as follows \cite{evans2002asymptotic}: for all $0< \rho < 1/d$, we have that as $n\rightarrow \infty$,
% \[
% \E\{\min_{1\le i \le n} c(x,x_i)\} = \frac{\Gamma(1+1/d)}{V_d^{1/d}n^{1/d}} \int_\Omega g(x) f(x)^{-1/d}dx + \mathcal O\left(\frac{1}{n^{1/d + (1/d-\rho)}}\right).
% \]
%We observe that the order $\mathcal O(n_k^{\scriptscriptstyle -1/d})$ can be attained, for example, by letting the data-generating distributions being the uniform distributions on $\Omega$. 

\begin{proof}[Proof of Lemma~\ref{lem:min_dist_unbound}]
The proof is based on a conditioning argument following \cite[Theorem~2.1]{penrose2003weak}, \cite[Theorem~2]{wade2007explicit}, and \cite[Theorem~2.3]{penrose2011laws}. Let $x$ be the random variable with density function $g$, and $\mathcal X_n=\{x_1,\ldots,x_n\}$ be a set of $n$ i.i.d. samples from the distribution $f$. Denote $\xi(x;\mathcal X_n)$ as the minimum distance from $x$ to the $n$ points within $\mathcal X_n$. 
For any fixed $x$, it has been shown in Lemma~3.2 of \cite{penrose2003weak} that the expectation of $n^{1/d}\{\min_{1\le i \le n} \left\Vert x-x_i\right\Vert\}$ converges to $\E[\xi_\infty(\mathcal P_{f(x)})]$, where $\mathcal P_{f(x)}$ is a homogeneous Poisson point process of intensity $f(x)$ (which is a constant when we fix $x$) on $\R^d$, $\xi_\infty(\mathcal P_{f(x)})$ is the limit of $\xi$ on $\mathcal P_{f(x)}$, see details in equation~(2.4) of \cite{penrose2003weak}, and here the expectation is taken with respect to the Poisson point process $\mathcal P_{f(x)}$. The expectation $\E[\xi_\infty(\mathcal P_{f(x)})]$ equals to $V_d^{-1/d}\Gamma(1+1/d)f(x)^{-1/d}$ as shown in equation~(15) of \cite{wade2007explicit}. 
If for random $x$, the function $\xi$ satisfies the moments condition $\sup_n \E[(n^{1/d}\xi(x;\mathcal X_n))^{1+\epsilon}]<\infty$ for some $\epsilon>0$, where the expectation is taken with respect to both the random variables $x$ and random samples in $\mathcal X_n$, then we can condition on the distribution of $x$ and thus as $n\rightarrow\infty$, $\lim_{n\rightarrow\infty}\E[n^{1/d}\xi(x;\mathcal X_n)]=\int_{\R^d}\E[\xi_\infty(\mathcal P_{f(x)})] g(x)dx$. By plug in the close-form value of $\E[\xi_\infty(\mathcal P_{f(x)})]$, we obtain the desired result.
\end{proof}

Therefore, under the assumption that $\lim_{n_1,n_2\to\infty} n_2/n_1= c>0$, in the asymptotic regime we roughly have $n_2 \sim c n_1$ for fixed constants $c>0$, and the expectation of the minimum distance is asymptotically of the order $\mathcal O(n_1^{-1/d})$ if we do not impose any further conditions for the data-generating distributions. 
%Similarly, when $m=2$, we obtain the variance is asymptotically $\mathcal O(N^{\scriptscriptstyle -2/d})$. 
%
Observe that in our case the first term on the right-hand side of \eqref{eq:bound_dist} is a variant of the minimum distance in Lemma \ref{lem:min_dist_unbound} in terms that we restrict our attention to points in the subset $\hOmega_2^\circ$. Therefore, by restricting the support of the integral we have that:
\[
(c n_1)^{1/d}\E_{\homega_1\sim P_1^\circ}[\ind_{\{\homega_1\in\Omega_2^\circ\}}\min_{j:\;\homega_2^j \in\Omega_2^\circ}c(\homega_1,\homega_2^j)] \rightarrow \frac{\Gamma(1+1/d)}{V_d^{1/d}} \int_{\Omega_2^\circ} \frac{f_{1}(x)}{[f_{2}(x)]^{1/d}}dx,  
%\int_{\Omega_2^\circ} \big\{\min_{j:\;\homega_2^j \in\Omega_2^\circ}c(\homega_1,\homega_2^j)\big\}f_1(\homega_1)d\homega_1 .
\]
and the second term on the right-hand side of \eqref{eq:bound_dist} can be treated similarly:
\[
(n_1)^{1/d}\E_{\homega_2\sim P_2^\circ}[\ind_{\{\homega_2\in\Omega_1^\circ\}}\min_{i:\;\homega_1^i \in\Omega_1^\circ}c(\homega_2,\homega_1^i)] \rightarrow \frac{\Gamma(1+1/d)}{V_d^{1/d}} \int_{\Omega_1^\circ} \frac{f_{2}(x)}{[f_{1}(x)]^{1/d}}dx.
\]
Then we apply a refined law-of-large-numbers-type argument to show the desired results in the theorem. Here we observe the key challenge is that the sample average is taken for dependent random variables. In particular, for fixed sample points $\{\homega_2^1,\ldots,\homega_2^{n_2}\}$ and two i.i.d. observations $\homega_1$ and $\homega_1'$, the variables $\min_{j:\;\homega_2^j \in\Omega_2^\circ}c(\homega_1,\homega_2^j)$ and $\min_{j:\;\homega_2^j \in\Omega_2^\circ}c(\homega_1',\homega_2^j)$ are {\it dependent} since they rely on the common sample set $\{\homega_2^1,\ldots,\homega_2^{n_2}\}$. To address this issue, we can apply the coupling argument used in the proof to \cite[Theorem~2.1]{penrose2003weak}. More specifically, for fixed $\homega_1$ and $\homega_1'$, we can separate the space into two half-spaces: $F_{\homega_1}$ that contains all points closer to $\homega_1$ than to $\homega_1'$, and $F_{\homega_1'}$ that contains all points closer to $\homega_1'$ than to $\homega_1$. Given these two half-spaces, we can construct two independent homogeneous Poisson process of intensity $f_1(\homega_1)$ and $f_1(\homega_1')$, respectively. Then by the coupling argument, it was shown in \cite[Proposition~3.2]{penrose2003weak} that we have the weak law of large numbers in the sense that 
\begin{equation}\label{eq:LLN1}
\frac{1}{n_1} \sum_{i:\;\homega_1^i\in\Omega_2^\circ} (c n_1)^{1/d}\min_{j:\;\homega_2^j \in\Omega_2^\circ}c(\homega_1^i,\homega_2^j) \rightarrow \frac{\Gamma(1+1/d)}{V_d^{1/d}} \int_{\Omega_2^\circ} \frac{f_{1}(x)}{[f_{2}(x)]^{1/d}}dx, \quad \text{in $L^1$},
\end{equation}
and 
\begin{equation}\label{eq:LLN2}
\frac{1}{n_2} \sum_{j:\;\homega_2^j\in\Omega_1^\circ}( n_1)^{1/d} \min_{i:\;\homega_1^i \in\Omega_1^\circ}c(\homega_2^j,\homega_1^i) \rightarrow  \frac{\Gamma(1+1/d)}{V_d^{1/d}}  \int_{\Omega_1^\circ} \frac{f_{2}(x)}{[f_{1}(x)]^{1/d}}dx, \quad \text{in $L^1$}.
\end{equation}
Finally, note that for any sequences of random variable $X_n,Y_n$, if $X_n \rightarrow X$ in $L^1$ and $Y_n \rightarrow Y$ in $L^1$, then $X_n + Y_n \rightarrow X+ Y$ in $L^1$. 
% Indeed, we have $\E|X_n + Y_n - (X+Y)| \leq \E|X_n - X|  + \E|Y_n - Y|$, thus $\E|X_n - X|\rightarrow 0$ and $\E|Y_n - Y|\rightarrow 0$ together implies that $\E|X_n + Y_n - (X+Y)| \rightarrow 0$. 
By combining \eqref{eq:LLN1} and \eqref{eq:LLN2}, we prove the theorem.

\section{Auxiliary Results}

\subsection{Kantorovich duality}\label{sec:kan}
\begin{lemma}[Theorem 5.10, \cite{villani2008optimal}]\label{lem:kan}
Let $(\mathcal X, \mu)$ and $(\mathcal Y, \nu)$ be two Polish probability spaces and let $c: \mathcal X \times \mathcal Y \rightarrow \mathbb R \cup \{ +\infty\}$ be a lower semicontinuous cost function. Then we have the duality
\[
\min_{\gamma \in \Pi(\mu,\nu)} \int_{\mathcal X \times \mathcal Y} c(x,y)d\gamma(x,y) = \sup_{\substack{ (\phi,\psi)\in L^1(\mu) \times L^1(\nu) \\ \phi(x) + \psi(y) \leq c(x,y)\\ \forall x,y}} \left( \int_{\mathcal X} \phi(x)d\mu + \int_{\mathcal Y} \psi(y)d\nu \right),
\]
where $\gamma \in \Pi(\mu,\nu)$ denotes the joint distribution on $\mathcal X \times \mathcal Y$, with marginal distributions $\mu$ and $\nu$, respectively. 
\end{lemma}
Note that when $\mu$ (or $\nu$) is a discrete distribution on $\{x_1,\ldots,x_{m}\}$, then the function $\phi(x)$ can be viewed as a vector $\xi \in \mathbb R^m$. And the above dual formulation will be reduced to 
\begin{equation}\label{eq:kan_discrete}
\min_{\gamma \in \Pi(\mu,\nu)} \int_{\mathcal X \times \mathcal Y} c(x,y)d\gamma(x,y) = \sup_{\substack{ \xi \in \mathbb R^m, \psi\in L^1(\nu) \\ \xi_i + \psi(y) \leq c(x_i,y)\\ \forall 1\leq i \leq m,\forall y}} \left( \sum_{i=1}^m \xi_i \mu(x_i) + \int_{\mathcal Y} \psi(y)d\nu \right),
\end{equation}
this is what we have used in the proof of Lemma~\ref{lemma:wass vs wass}.

\subsection{Interchangeability principle}
Before introducing the principle, we recall the definition for decomposable spaces. Assume a probability space $(\Omega,\mathcal F, P)$. A linear space $\mathcal M$ of $\mathcal F$-measurable functions $\psi:\Omega \rightarrow \mathbb R^d$ is decomposable if for every $\psi \in \mathcal M$ and $A\in\mathcal F$, and every bounded $\mathcal F$-measurable function $\phi:\Omega \rightarrow \mathbb R^d$, the space $\mathcal M$ also contains the function $\eta(\cdot) = \ind_{\Omega\setminus A}(\cdot)\psi(\cdot) + \ind_{A}(\cdot)\phi(\cdot)$. 
\begin{lemma}[Theorem 7.80, \cite{shapiro2009lectures}]\label{lem:change}
Let $\mathcal M$ be a decomposable space and $f: \mathbb R^d \times \Omega \rightarrow \mathbb R$ be a random lower semicontinuous function. Then
\begin{equation}\label{eq:change}
\mathbb E\left[ \inf_{x\in \mathbb R^d} f(x,\omega) \right] = \inf _{\chi \in \mathcal M} \mathbb E[F_{\chi}],
\end{equation}
where $F_{\chi}:=f(\chi(\omega),\omega)$, provided that the right-hand side of \eqref{eq:change} is less than $+\infty$. Moreover, if the common value of the both sides in \eqref{eq:change} is not $-\infty$, then
\[
\tilde \chi \in \argmin_{\chi \in \mathcal M} \mathbb E[F_{\chi}] \text{  iff  } \tilde \chi(\omega) \in \argmin_{x \in \mathbb R^d} f(x,\omega) \text{ for a.e. $\omega\in\Omega$ and $\tilde \chi \in \mathcal M$}.
\]
\end{lemma}

\subsection{Wasserstein distributionally robust optimization}
The following result is a special case of Theorem 1 in \cite{gao2016distributionally} by choosing the Wasserstein metric of order $1$.
\begin{lemma}[Theorem 1, \cite{gao2016distributionally})]\label{lemma:DRO}
For $\nu = \frac1N\sum_{i=1}^N \delta_{\hat\xi^i}$ %, $p\in[1,+\infty)$ 
and $\theta>0$, we have:
\[
\sup_{\mu\in\calP(\nu,\theta)}\int_\Omega \Phi(\xi)d\mu(\xi) = \min_{\lambda\geq0}\left\{\lambda \theta -\frac1N \sum_{i=1}^N \inf_{\xi\in\Omega}[\lambda c (\hat\xi^i,\xi) - \Phi(\xi)]\right\},
\]
where $\calP(\nu,\theta)$ is the uncertainty set induced by Wasserstein metric \cite{gao2016distributionally} as defined in \eqref{problem:wass vs wass}:
\[
\calP(\nu,\theta) = \{\mu\in\calP(\Omega): \wass(\mu,\nu)\leq \theta\}.
\]
%\rgao{Only state $p=1$?}Done

\end{lemma}

\end{document}